\documentclass[10pt,a4paper,reqno]{amsart}
\usepackage{amsmath,amssymb,amsfonts,amsthm,amsopn}
\usepackage{latexsym,graphicx}
\usepackage{xcolor}
\usepackage{color, colortbl}
\usepackage{pb-diagram}
\usepackage[title]{appendix}
\usepackage{tikz}
\usepackage{hyperref}
\usepackage{enumerate}
\usepackage[normalem]{ulem}
\usepackage{cancel}

\newcommand{\tends}[1]{\mbox{\space \raise-2mm
\hbox{$\textstyle\longrightarrow\atop\scriptstyle {#1}$} \space}}

\newcommand{\mcF}{\mathcal F}
\newcommand{\mcL}{\mathcal L}

\newcommand{\mcU}{\mathcal{U}}

\newcommand{\sinc}{\operatorname{sinc}}

\newcommand{\R}{\mbox{${\mathbb{R}}$}}
\newcommand{\N}{\mbox{${\mathbb{N}}$}}
\newcommand{\T}{\mbox{${\mathbb{T}}$}}
\newcommand{\Z}{\mbox{${\mathbb{Z}}$}}
\newcommand{\C}{\mbox{${\mathbb  C}$}}

\newcommand{\la}{\lambda_{\el}}

\newcommand{\Chi}{\mbox{\large${\chi}$}}

\newtheorem{theorem}{Theorem}[section]
\newtheorem{lemma}[theorem]{Lemma}
\newtheorem*{lemma*}{Lemma} 

\newtheorem{proposition}[theorem]{Proposition}
\newtheorem{definition}[theorem]{Definition}

\newtheorem{remark}[theorem]{Remark}
\theoremstyle{definition}
\newtheorem{example}[theorem]{Example}

\newcommand{\beqa}{\begin{eqnarray*}}
    \newcommand{\eeqa}{\end{eqnarray*}}

\newcommand{\field}[1]{\mathbb{#1}}
\newcommand{\bN}{\field{N}}        
\newcommand{\bZ}{\field{Z}}        
\newcommand{\bC}{\field{C}}        
\newcommand{\cL}{\mathcal{L}}     %

\newcommand{\bh}{{\mathbf h}}

\def\la{\lambda}

\def\cF{\mathcal{F}}              

\def\cH{\mathcal{H}}
\def\cB{\mathcal{B}}

\def\cU{\mathcal{U}}

\def\l{\langle}
\def\r{\rangle}
\def\<{\left<}
\def\>{\right>}

\def\mv1{M_v^1}

\def\mn{(m,n)}
\def\mn'{(m',n')}

\def\ds{\rule{0pt}{1.5ex}}

\newcommand{\rev}[1]{{\color{red}#1}}

\def\Ren{\mathbb{R}^d}

\def\f{\varphi}

\def\Sn2{S_{2}(L^{2}(\Ren))}
\def\S1{S_{1}(L^{2}(\Ren))}
\def\sig00{\sigma_{0,0}}

\def\la{\langle}
\def\ra{\rangle}

\def\bk{{\bf{k}}}
\def\bh{{\bf{h}}}

\begin{document}

\begin{abstract}
We study dynamical sampling for graph signals using the spectral theory of the
normalized graph Laplacian. Graph Paley-Wiener spaces \(GPW_\omega\) are
defined as spectral subspaces associated with a bandwidth parameter
\(\omega\), and reconstruction is studied from measurements generated by
iterates of bounded operators leaving these spaces invariant.

For the lattice graph \(\mathbb Z\), we give an explicit spectral
decomposition of the normalized Laplacian, identify the corresponding
Graph Paley--Wiener spaces with classical bandlimited spaces on the torus, and
recover the classical sampling theorem in this setting. Using frames of
operator orbits, we construct dynamical frames for \(GPW_\omega(\mathbb Z)\)
and derive reconstruction formulas. We also discuss the extension to
\(\mathbb Z^n\), where the bandwidth region is described by a sublevel set of
the Laplacian symbol.

 Finally, for finite graphs, we formulate dynamical recovery in
\(GPW_\omega(G)\) in terms of finite frames generated by iterates of the
graph operator. Using the path graph \(P_N\) as a numerical model, we study
how the number of iterates and the spatial distribution of the sampled
vertices affect the optimal frame bounds and the conditioning of the
resulting dynamical frame.
\end{abstract}

\title{
 Dynamical Sampling in Graph Paley--Wiener Spaces}
\author[I. M. Bulai]{Iulia Martina Bulai}
\address{Universit\`a di Torino, Dipartimento di Matematica, via Carlo Alberto 10, 10123 Torino, Italy}
\email{iuliamartina.bulai@unito.it}

\author[C. Cabrelli]{Carlos Cabrelli}
\address{Department of Mathematics, FCEyN-University of Buenos Aires and IMAS-CONICET Buenos Aires, Argentina}
\email{cabrelli@dm.uba.ar}

\author[E. Cordero]{Elena Cordero}
\address{Universit\`a di Torino, Dipartimento di Matematica, via Carlo Alberto 10, 10123 Torino, Italy}
\email{elena.cordero@unito.it}

\author[U. Molter]{Ursula Molter}
\address{Department of Mathematics, FCEyN-University of Buenos Aires and IMAS-CONICET Buenos Aires, Argentina}\email{umolter@dm.uba.ar}

\author[S. Saliani]{Sandra Saliani}
\address{Universit\`a di Napoli Parthenope, Dipartimento di Ingegneria, Centro Direzionale Isola C4, 80143 Napoli, Italy}
\email{sandra.saliani@uniparthenope.it}
\date{}

\maketitle

\section{Introduction}

The classical Shannon sampling theorem provides a complete characterization
of bandlimited functions on the real line in terms of their pointwise
samples. Since its introduction, this result has served as the foundation of
sampling theory and has inspired numerous generalizations to more abstract
settings. In particular, considerable attention has recently been devoted to
sampling problems for signals defined on graphs, where the underlying graph
structure replaces the Euclidean geometry of the classical theory
\cite{AldroubiBaileyKrishtalMillerPetrosyan2024,Pesenson2001BLV,PEsensonTAMS2008,PEsensonTAMSerrata,PesensonPesenson2010,PesensonPesenson2021}.
Such developments are motivated by
applications in graph signal processing, where data are naturally associated
with the vertices of a graph and reconstruction algorithms must exploit the
topological and spectral properties of the network; see, for example,
\cite{BCPS2026}, \cite{BulaiSaliani2023}, \cite{SRV}.

An important extension of graph sampling theory incorporates the temporal
evolution of the signal. Instead of observing a graph signal only once, one
collects samples at several time instants while the signal evolves according
to a linear graph operator. This setting, known as \emph{dynamical sampling},
allows one to compensate for spatial undersampling by exploiting temporal
redundancy. Over the last decade, dynamical sampling has become an active
research area at the intersection of sampling theory, frame theory, operator
theory, and inverse problems; see, for instance,
\cite{ACNP25,ACKM2026} and the references therein.

The mathematical formulation of dynamical sampling is closely related to the
theory of frames generated by the iterates of a bounded operator, often
called \emph{dynamical frames}. Given an operator $T$ acting on a Hilbert
space, one seeks conditions under which the family
\[
\{T^n f_i:\; i\in I,\; n\ge0\}
\]
forms a frame, thereby providing stable reconstruction formulas from
spatio-temporal measurements. This point of view has led to a rich operator
theory with deep connections to spectral theory and has proved particularly
well suited for the study of evolving graph signals.

The purpose of this paper is to develop this dynamical framework in the
setting of graph signal processing. More precisely, we study the recovery of
bandlimited graph signals from samples obtained by the successive iterates of a graph operator. Our approach relies on the
theory of dynamical frames and is inspired by the operator-theoretic methods
introduced in \cite{UCACHA2017}. For the lattices $\bZ$ and $\bZ^n$ we
give explicit constructions of frames of iterations for the Graph Paley--Wiener
spaces, while for finite graphs we characterize stable recovery in terms of the
spectral properties of a graph operator and of the sampling set, and we
quantify the resulting frame bounds.

 We work with graphs \( G = (V(G), E(G)) \) that may be finite or countably infinite, and are assumed to be connected. Here, \( V(G) \) represents the set of vertices, while \( E(G) \) denotes the set of edges. Throughout, we restrict our attention to \emph{simple}, \emph{undirected}, and \emph{unweighted} graphs, meaning that there are no loops or multiple edges between vertices.

For any vertex \( v \in V(G) \), the number of vertices directly connected to \( v \) (i.e., its neighbors) is called the \emph{degree} of \( v \), denoted \( d(v) \). We assume that the degrees of all vertices are uniformly bounded above, and we define the \emph{maximum degree} of the graph by
\[
d(G) = \sup_{v \in V(G)} d(v).
\]

The space \( \ell^2(G) \) is the Hilbert space consisting of all complex-valued functions \( f : V(G) \to \mathbb{C} \), equipped with the inner product
\[
\langle f, g \rangle = \sum_{v \in V(G)} f(v) \, \overline{g(v)}
\]
and the  induced norm
\begin{equation} \label{eq:norm}
    \|f\| = \|f\|_2 = \left( \sum_{v \in V(G)} |f(v)|^2 \right)^{1/2}.
\end{equation}

The \emph{discrete {\bfseries normalized} Laplace operator} \( {\mcL} \) acting on functions \( f \in \ell^2(G) \) is defined by
\begin{equation} \label{eq:laplacian}
    {\mcL} f(v) = \frac{1}{\sqrt{d(v)}} \sum_{u \sim v} \left( \frac{f(v)}{\sqrt{d(v)}} - \frac{f(u)}{\sqrt{d(u)}} \right),
\end{equation}
where \( u \sim v \) indicates that vertices \( u \) and \( v \) are adjacent (i.e., connected by an edge).
It is known that \( {\mcL} \) defines a bounded, self-adjoint, and positive semidefinite operator on \( \ell^2(G) \), cf. \cite{PEsensonTAMS2008}; see also \cite{Olebook} for the frame-theoretic notions used below. \par
Let \( \sigma({\mcL}) \subset \R_{\geq 0}\) denote the spectrum of \( {\mcL} \). Define
\[
\lambda_{\min} = \inf_{\lambda \in \sigma({\mcL})} \lambda, \qquad \lambda_{\max} = \sup_{\lambda \in \sigma({\mcL})} \lambda.
\]

According to the spectral theorem for self-adjoint operators \cite{BS1987}, there exists a direct integral of Hilbert spaces
(see also \cite{Pesenson2001BLV} for the analogous construction of band-limited
vectors associated with a general self-adjoint, positive-definite operator)
\[
X = \int_{\sigma(\mcL)}^\oplus X(z) \, d\mu(z),
\]
where  $z \in \sigma(\mcL),$ $(X(z),  \|\;\;\|_{z})$ are Hilbert spaces, $\mu$ is the associate scalar spectral measure, and
$X$ is the set of all $x=\{x(z)\}_{z\in\sigma(\mcL)}$ with $x(z ) \in X(z) \text {  for all } z \in \sigma(\mcL)$ such that
the norm
\begin{equation*}\label{eq:normX}
    \|x\|_{X}^2 \equiv  \int_{\sigma(\mcL)} \|x(z)\|_{z}^2 \, d\mu(z) < +\infty.
\end{equation*}

Moreover, there exists  a unitary operator
\begin{equation}\label{calU}
    \mathcal{U} : \ell^2(G) \to X,
\end{equation}
such that  $\mathcal{U}$ diagonalizes the operator \({\mcL} \). Namely, if $M_z: X\to X$ is the multiplication operator by the independent variable i.e., \begin{equation}\label{eq:Multip}
M_z f(z)=z f(z),
\end{equation} we have
\begin{equation}\label{Fspettr}
    \mathcal{U}{\mcL} = M_z \mathcal{U}.
\end{equation}

This motivates the following definition.

\begin{definition}\label{DefPW}
    Let \( \omega \in [\lambda_{min},\lambda_{max}]\). A function \( f \in \ell^2(G) \) is said to belong to the \emph{Graph Paley-Wiener space} \( GPW_\omega(G) \) if its ``graph Fourier transform'' \( \mathcal{U} f \) is supported in the interval \( [0, \omega] \); that is,
    \[
    \operatorname{supp}(\mathcal{U} f) \subseteq [0, \omega].
    \]
\end{definition}

These spaces play the role of bandlimited functions on the graph.

Since the operator \( {\mcL} \) is bounded, every function \( f \in \ell^2(G) \) is contained in some Graph Paley-Wiener space \( GPW_\omega(G) \) for a suitable value \( \omega \in \sigma({\mcL}) \).  Hence,
\begin{equation} \label{eq:pw-stratification}
   \ell^2(G) = GPW_{\lambda_{\max}}(G) = \bigcup_{\omega \in \sigma({\mcL})} GPW_\omega(G), 
    \end{equation}
$$\text{ with }\quad GPW_{\omega_1}(G) \subseteq GPW_{\omega_2}(G), \quad \text{for all } \omega_1 < \omega_2.$$



This structure allows us to view the space \( \ell^2(G) \) as a union of nested Graph Paley-Wiener spaces corresponding to increasing spectral bandwidths.

Various analytical properties of the Graph Paley-Wiener spaces \( GPW_\omega(G) \) can be established, including a generalized version of the Classical Sampling Theorem adapted to the graph setting.
The latter allows us to recover $f \in
GPW_\omega(G)$ from  pointwise evaluations on a
coarser lattice, which can be advantageous in some
applications. These evaluations are the coefficients of $f$ with respect
to an orthonormal basis; such a system is perfectly conditioned, but it offers
no redundancy: the loss or corruption of a single sample cannot be compensated
by the remaining ones. See for example \cite{H1996} (chapter 11).

In practice, exact pointwise values are
rarely available, since measurements are inevitably
affected by noise. It is thus more realistic to assume
that measurements consist of inner products between the
unknown function and the elements of a suitable frame.
Redundant frames offer greater robustness to noise and
data loss, allowing for a more reliable reconstruction
of the unknown function from such measurements.

In this paper, we show how to construct
\textit{frames of iterations} of $GPW_{\omega}(\mathbb{Z})$,
using tools from dynamical sampling.



 The key problem is therefore to determine conditions for the existence of
frames of iterations for the Graph Paley--Wiener spaces $GPW_\omega(G)$: a
signal evolving under an evolution operator is typically observed only at a
coarse, spatially insufficient set of sampling locations, and the idea of
dynamical sampling is to compensate for this by collecting samples at the
same locations over multiple time instances, exploiting spatial and
temporal information jointly. This can be reformulated as determining when
the orbit of an operator generates a frame, and it leads naturally to the
construction of frames of iterations described above for $G=\bZ$ and
$G=\bZ^n$ (Section~\ref{sec:4}).
\vspace{0.1truecm}

Let us describe in more detail the contributions of the paper.
\vspace{0.1truecm}

1) For the lattice graph $\bZ$, we obtain an explicit spectral
	decomposition of the normalized graph Laplacian $\cL$ in terms of a direct
	integral of two-dimensional Hilbert spaces (Theorem~\ref{teospectr}), from
	which the Graph Paley--Wiener spaces $GPW_\omega(\bZ)$ are identified with
	classical bandlimited spaces $L^2(\Omega)$ on the torus, for a
	band $\Omega$ determined by $\omega$. As a first consequence, we recover
	the Classical Sampling Theorem on $\bZ$ as a particular case of this
	identification (Theorem~\ref{ShannotheoremZ}): every $f\in GPW_\omega(\bZ)$
	with $\omega=\varphi(1/2\ell)$ can be reconstructed from its values on the
	sub-lattice $\ell\bZ$ via a sinc-interpolation formula.
\vspace{0.1truecm}

2)	 Building on this spectral identification and on the characterization
	of frames of iterations of a normal operator recalled in
	Theorem~\ref{TeorCarl} (from \cite{UCACHA2017}), we construct explicit
	\emph{dynamical frames} for $GPW_\omega(\bZ)$: for every uniformly separated
	sequence $\Lambda=(\lambda_j)_{j}$ in the unit disc and every suitable
	$a\in L^2(\Omega)$ we build a bounded operator $R$ on $\ell^2(\bZ),$ leaving $GPW_\omega(\bZ)$ invariant, such that $(R^nf_0)_{n\ge0}$ is a frame for
	$GPW_\omega(\bZ)$ (Theorems~\ref{main} and~\ref{main2}), together with the
	associated reconstruction formula \eqref{RC}. This shows that graph signals
	in $GPW_\omega(\bZ)$ can be stably recovered from purely dynamical
	measurements $\{\langle f, R^nf_0\rangle\}_{n\ge0}$, obtained by
	testing the evolved signal $(R^*)^nf$ against one fixed sensor $f_0$ at every
	time step. 

\vspace{0.1truecm}
3) We extend this construction to the lattice $\bZ^n$, where the
	Graph Paley--Wiener spaces are described by a sublevel set $\Omega_n$ of the
	Laplacian symbol $\tfrac{2}{n}\sum_{j=1}^n\sin^2(\pi\xi_j)$. Since $\Omega_n$
	need not be a nice domain for general $\omega$, we develop a general method
	that transports an orthonormal basis of a simpler reference domain $D$ onto
	$\Omega_n$ through a piecewise translation, and we carry
	out this construction explicitly for $n=2$ and $\omega=1$, where
	$\Omega_2$ is the diamond $|\xi_1|+|\xi_2|\le 1/2$.
	\vspace{0.1truecm}
	
	4)  We formulate a finite-dimensional counterpart of the whole theory
	(Section~\ref{sec:finite}). For a finite graph $G$ with normalized
	Laplacian $\mathcal L$ and dynamics $R=h(\mathcal L)$, we express the
	dynamical samples of a signal in $GPW_\omega(G)$ as the rows of an explicit
	analysis matrix $\Theta_{S,K}$, and we show that stable recovery is
	equivalent to a finite-frame condition for this matrix, with optimal
	frame bounds $A_{S,K}$, $B_{S,K}$ given by its extremal singular values.
	
	 Using the path graph $P_N$ as a numerical model, we quantify how the
	number of iterates $K$ and the spatial distribution of the sampling set $S$
	affect $A_{S,K}$, $B_{S,K}$, and the condition number
	$B_{S,K}/A_{S,K}$. We show, in particular, that a spatially distributed
sampling set (a purely spatial benchmark) can produce an exactly tight frame already at $K=1$, whereas
sampling repeatedly at a single vertex reaches full rank only after $K=m=\dim GPW_\omega$ iterates, as predicted by Theorem~\ref{DS}, and remains, even then, several orders of magnitude worse
conditioned than sampling at two well-separated vertices with the same
total number of measurements. These experiments illustrate that, in the
	finite-dimensional setting, the spatial distribution of the sampling
	vertices is at least as important as the number of temporal samples for
	the stability of the reconstruction.

\vspace{0.1truecm}	

The paper is organized as follows. Section~\ref{sec:preliminaries} recalls
the basic notions on frames, the Classical Sampling Theorem, and frames of
operator orbits, including the characterization of Theorem~\ref{TeorCarl}
from \cite{UCACHA2017} on which our constructions rely.
Section~\ref{sec:4} is devoted to the case $G=\bZ$: we describe the spectral
decomposition of the normalized Laplacian, identify the corresponding Graph
Paley--Wiener spaces, recover the Classical Sampling Theorem, and construct
the dynamical frames and reconstruction formulas summarized above; the
extension to $\bZ^n$ is discussed in Section~\ref{subsec:zn}. Finally,
Section~\ref{sec:finite} treats finite graphs: it develops the
finite-dimensional analogue of dynamical frames via the analysis matrix
$\Theta_{S,K}$ and the finite-frame characterization of
Theorem~\ref{DS}, and illustrates the theory with the numerical
experiments on the path graph $P_N$ described above.

\section{Preliminaries}\label{sec:preliminaries}


    Let $\mathcal{H}$ be a separable Hilbert space and let $I$ be an at most countable index set.
    A sequence $\{f_n\}_{n \in I} \subset \mathcal{H}$ is called a \textbf{frame} for $\mathcal{H}$ if there exist constants $A, B > 0$ such that for all $f \in \mathcal{H}$,
    \[
    A \|f\|^2 \leq \sum_{n \in I} |\langle f, f_n \rangle|^2 \leq B \|f\|^2.
    \]
    The constants $A$ and $B$ are called the \emph{frame bounds}.
    Let $\{f_n\}_{n \in I}$ be a frame for $\mathcal{H}$. A sequence $\{g_n\}_{n \in I} \subset \mathcal{H}$ is called a \textbf{(canonical) dual frame} of $\{f_n\}$ if for every $f \in \mathcal{H}$,
    \[
    f = \sum_{n \in I} \langle f, f_n \rangle g_n = \sum_{n \in I} \langle f, g_n \rangle f_n.
    \]
    The canonical dual frame is given by $ g_n = S^{-1} f_n$,
    where $S$ is the frame operator defined by
    \[
    Sf = \sum_{n \in \mathbb{N}} \langle f, f_n \rangle f_n.
    \]

    \textbf{Reconstruction Formula}
    Given a frame $\{f_n\}_{n \in \mathbb{N}}$ and its canonical dual frame $\{g_n\}_{n \in \mathbb{N}}$, any vector $f \in \mathcal{H}$ can be reconstructed via
    \[
    f = \sum_{n \in I} \langle f, f_n \rangle g_n = \sum_{n \in I} \langle f, g_n \rangle f_n.
    \]

\subsection{Classical Sampling Theorem}
Let $u$ be a positive real number. A function \( f \in L^2(\mathbb{R}) \) belongs to the   Paley-Wiener class \( PW_u(\mathbb{R}) \) and is called \(u\)-\emph{bandlimited} if its \(L^2\)-Fourier transform
    \[
    \hat{f}(\xi) = \int_{-\infty}^{+\infty} f(x) e^{-2\pi i x \xi} \, dx
    \]
    has support in \([-u, u]\), i.e., \( \hat{f}(\xi) = 0 \) for almost all \( |\xi| > u \).

        \begin{theorem}[Paley-Wiener Theorem]
            Let \( f \in L^2(\mathbb{R}) \). The following statements are equivalent:
            \begin{enumerate}
                \item $f\in PW_u(\mathbb{R}).$

                \item The function \( f \) extends to an entire function on \( \mathbb{C} \) and satisfies the growth condition
                \[
                |f(z)| \leq C e^{2\pi u |\operatorname{Im}(z)|}, \quad \text{for all } z \in \mathbb{C},
                \]
                for some constant \( C > 0 \).
            \end{enumerate}
        \end{theorem}

Functions \(f\in {PW}_{u}(\mathbb{R}) \)  are often called \emph{Paley-Wiener functions}.
    \begin{theorem}[Classical Sampling Theorem (CST)]\label{ShannonT}
        Let \(f\in PW_u(\mathbb{R}) \).
        Then \( f \) is completely determined by its samples at the points \( \left\{ \frac{j}{2u} \right\}_{j \in \mathbb{Z}} \), and can be reconstructed by the formula
        \begin{equation}\label{shannon}
        f(x) = \sum_{j \in \mathbb{Z}} f\left( \frac{j}{2u} \right) \cdot \frac{\sin\left(2\pi
         u\left(x - \frac{j}{2u}\right)\right)}{2\pi u\left(x - \frac{j}{2u}\right)}= \sum_{j \in \mathbb{Z}} f\left( \frac{j}{2u} \right) \cdot \operatorname{sinc}\left(2u\left(x - \frac{j}{2u}\right)\right)
        \end{equation}
        where the \emph{sinc function} is defined by
        \begin{equation}\label{eq:sinc}
        \operatorname{sinc}(x) = \frac{\sin(\pi x)}{\pi x},\quad x\not=0,\,\,\mbox{    and }\,\,\operatorname{sinc}(0)=1.
        \end{equation}
 The series converges in the \(L^2(\mathbb{R})\)-sense and uniformly.
    \end{theorem}

        \subsection{Frames of orbits}

We conclude this section by recalling a result on the existence and
construction of frames of iterations for bounded \emph{normal} operators,
the class of operators considered throughout this paper (see
\cite{UCACHA2017}). For a general treatment of frames of iterations, we
refer the reader to \cite{ACNP25} and the references therein, while a
comprehensive survey of the subject is provided in \cite{ACKM2026}.

 We begin by introducing some notation and recalling a few definitions. Given two nonnegative sequences
\((a_j)_{j\in\mathbb{N}}\) and \((b_j)_{j\in\mathbb{N}}\), we write
\[
a_j \sim b_j
\]
if there exist constants \(c,C>0\) such that
\[
c\,a_j \le b_j \le C\,a_j,
\qquad j\in\mathbb{N}.
\]

A sequence \(\Lambda=(\lambda_j)_{j\in\mathbb{N}}\) in the open unit disk
\[
\mathbb{D}=\{z\in\mathbb{C}:|z|<1\}
\]
is said to be \emph{uniformly separated} if
\begin{equation}\label{carleson-cond}
\delta_\Lambda
:=
\inf_{j\in\mathbb{N}}
\prod_{k\neq j}
\left|
\frac{\lambda_j-\lambda_k}
{1-\overline{\lambda_k}\lambda_j}
\right|
>0.
\end{equation}
Condition~\eqref{carleson-cond} is known as the \emph{Carleson condition}.

Next, we recall Definition~12.10 from~\cite{GMRbook}.

\begin{definition}
A sequence \(\{\lambda_n\}_{n\ge1}\subset\mathbb{D}\) is called \emph{exponential} if there exists a constant
\(0<c<1\) such that
\[
1-|\lambda_{n+1}|
\le
c\bigl(1-|\lambda_n|\bigr),
\qquad n\ge1.
\]
\end{definition}

We shall also use the following result (see Proposition~12.11 in~\cite{GMRbook}).

\begin{proposition}
Every exponential sequence is uniformly separated.
\end{proposition}
The previous result allows to construct uniformly separated sequences easily, as shown below.
    \begin{example}\label{ex:unif}
        An example of uniformly separated sequence is given by $(\lambda_n)_{n\geq 1}$, with $\lambda_n=1-e^{-n}$. Then $\lambda_{n+1}=1-e^{-n-1}$, and
        $$1-\lambda_{n+1}=e^{-n-1}=\frac{1}{e}e^{-n}=\frac{1}{e}(1-\lambda_n),$$
        as desired.
    \end{example}

   The following result will be the key tool for our main result. The
characterization of frames of iterations of normal operators was obtained in
\cite[Theorem~3.16]{UCACHA2017} and \cite{AP17} (see \cite{CabrelliMolterPaternostroPhilipp} for a more
general class of operators); the explicit frame bounds below are due to
\cite[Theorem~5.1]{CF2019}. Here $\bN=\{0,1,2,\dots\}$.
\begin{theorem}[Theorem 5.1 in \cite{CF2019}]\label{TeorCarl}
        Let \( T \in \mathcal{B}(\mathcal{H}) \) be a normal operator and let \( f_0 \in \mathcal{H} \). Then \( (T^n f_0)_{n \in \mathbb{N}} \) is a frame for \( \mathcal{H} \) if and only if
        \begin{equation}\label{diag}
        T = \sum_{j=0}^\infty \lambda_j \langle \cdot, e_j \rangle e_j,
        \end{equation}
        where \( (\lambda_j)_{j \in \mathbb{N}} \subset \mathbb{D} \) is uniformly separated, \( (e_j)_{j \in \mathbb{N}} \) is an orthonormal basis for \( \mathcal{H} \), and
        \begin{equation}\label{coeff}
        |\langle f_0, e_j \rangle|^2 \sim 1 - |\lambda_j|^2.
        \end{equation}
        In this case, the frame \( (T^n f_0)_{n \in \mathbb{N}} \) has frame bounds \( \alpha \Delta^{-1} \) and \( \beta \Delta \), where
        \[
        \alpha := \inf_{j \in \mathbb{N}} \frac{|\langle f_0, e_j \rangle|^2}{1 - |\lambda_j|^2}, \qquad
        \beta := \sup_{j \in \mathbb{N}} \frac{|\langle f_0, e_j \rangle|^2}{1 - |\lambda_j|^2}, \qquad
        \Delta := \frac{2}{\delta_\Lambda^4} \left(1 - 2 \log \delta_\Lambda \right)
        \]
        and $\delta_\Lambda$ is the uniform bound defined in \eqref{carleson-cond}.
\end{theorem}
    \begin{remark}
        Observe that condition \eqref{coeff} and $f_0\in\cH$ imply
    $$1 - |\lambda_j|^2\sim |\langle f_0, e_j \rangle|^2  \to 0, \,\qquad j\to\infty$$
    hence $|\lambda_j|\to 1$  for  $j\to\infty.$
    \end{remark}

    \section{The case \texorpdfstring{$G=\bZ$}{G=Z}}\label{sec:4}
    \subsection{The spectral decomposition of the normalized Laplacian}

    Consider the case of the Cayley graph of the additive abelian group \( \mathbb{Z} \). By abuse of notation we shall write $G=\bZ.$
In this setting $d(v)=2$, for every $v\in\bZ$ so that $d(G)=2$.
Using \eqref{eq:laplacian} we can compute the normalized Laplacian of a function $f \in \ell^2(\bZ)$, given by
\begin{equation}\label{eq:normlap}
\mathcal{\mcL}f(n)=\frac{1}{\sqrt{2}}\sum_{m\sim n} \left(\frac{f(n)}{\sqrt{2}}-\frac{f(m)}{\sqrt{2}}\right)=\frac12 \left(2f(n)-f(n+1)-f(n-1)\right),\quad n\in\bZ.
\end{equation}
    The Pontryagin dual of the  group \( \mathbb{Z} \) is the one-dimensional torus \( \mathbb{T} \).
    We use here the standard identification of $\T$ with the additive group $[-1/2,1/2)$ with the sum in $\R$ modulus one.
    From now on we set $\T\equiv [-1/2,1/2).$

    Accordingly, the Fourier transform \( \mathcal{F} \) on the Hilbert space \( \ell^2(\mathbb{Z}) \) is defined by
\begin{equation}\label{FT}
    \mathcal{F}f(\xi) = \sum_{k \in \mathbb{Z}} f(k) e^{2\pi i k \xi}, \quad f \in \ell^2(\mathbb{Z}), \; \xi \in \T.
\end{equation}

This transform defines a unitary operator from \( \ell^2(\mathbb{Z}) \) onto \( L^2(\T) := (L^2(\T), d\xi ) \), where \( d\xi \) is the normalized Haar measure.
We emphasize that functions in $L^2(\T)$ are $1$-periodic.
        \begin{lemma}\label{lemma0}
    The normalized Laplace operator \( \mathcal{L} \) in \eqref{eq:normlap} is a convolution operator. Thus, on the  Fourier transform side, it is a multiplication operator:
    \begin{equation}\label{L-multip}
    \mathcal{F}\mathcal{L} \mathcal{F}^{-1} = M_{\f},
    \end{equation}
    where  $M_{\f} : L^2(\T) \rightarrow L^2(\T)$ is the multiplication operator by the function
     \begin{equation}\label{fi}
        \f(\xi)=(1-\cos(2\pi \xi))=2\, \sin^2({\pi\xi}),\quad\xi\in\ \T.
    \end{equation}
  Namely, $ M_{\f} u=\f \, u$, for every $u\in L^2(\T)$.

    \end{lemma}
    \begin{proof}
    Using \eqref{eq:normlap} and \eqref{FT}, we compute
    \begin{align*}
        \mathcal{F}(\mathcal{L} f)(\xi)&=\frac{1}{{2}}\left(2\sum_{n\in\bZ} f(n)e^{2\pi i n\xi}-\sum_{n\in\bZ}f(n+1)e^{2\pi i n\xi}-\sum_{n\in\bZ}f(n-1)e^{2\pi i n\xi}\right)\\
        &= \cF{f}(\xi)-\frac12 e^{-2\pi i \xi} \cF{f}(\xi)-\frac12 e^{2\pi i \xi} \cF{f}(\xi)\\
        &=\cF{f}(\xi)\left(1-\frac12\left( e^{-2\pi i \xi}+e^{2\pi i \xi}\right)\right)\\
        &=(1-\cos(2\pi \xi))\cF{f}(\xi).
    \end{align*}
   This is the desired equality.
    \end{proof}
\begin{remark} If we replace the graph $G = \mathbb{Z}$ with another graph whose normalized Laplace operator can also be expressed as a convolution (as in the case of Cayley graphs of finitely generated groups), then our results will remain valid for that graph as well, with the appropriate modifications.
    \end{remark}
    Now, let
    \begin{equation}\label{X}
    X = L^2([0,2], \bC^2;\mu)=L^2([0,2],\mu)\oplus L^2([0,2],\mu),
    \end{equation}
    with measure $\mu$ having density
\begin{equation}\label{mu}
    d\mu(z):=\frac1{2\pi}\frac{dz}{\sqrt{z(2-z)}},
\end{equation}
and define the operator   $W :  L^2(\T) \longrightarrow X$ by
    \begin{equation}\label{W}
 (Wg)(z) = \left(g\left(\frac1{2\pi}\arccos(1-z)\right),g\left(1-\frac1{2\pi}\arccos(1-z)\;
    \right)\right),\quad g \in L^2(\T),\,
z\in [0,2).
     \end{equation}

     \begin{theorem}\label{teospectr}
Consider the normalized Laplace operator \( \mathcal{L} \) in \eqref{eq:normlap} and  the operator  $W$ defined in \eqref{W}. Then,  operator $ \cU:   \ell^2(\bZ)\to X,$
  defined by
  \begin{equation} \quad \cU= W\mcF,
      \end{equation}
      is a unitary operator.  Furthermore,
  \begin{equation}\label{spettU}
       \cL\,  =\cU^{-1} M_z \,\cU,
    \end{equation}
where $M_z: X \rightarrow X$ is the  multiplication operator by the independent variable defined in \eqref{eq:Multip}, i.e., $(M_z h)(z)=z h(z)$, $z\in [0,2).$
 \end{theorem}
\begin{proof}

We will first prove that the operator $W :  L^2(\T) \longrightarrow X$ is unitary, which will imply that  $ \cU:   \ell^2(\bZ) \to X$ is unitary.

Using Lemma \ref{lemma0} we have
    \begin{equation}\label{intertw}
        \cL =\cF^{-1}M_\f \cF.
    \end{equation}
    Since $\cF$ is unitary, the spectrum $\sigma(\cL)$ of $\cL$ is given by
    $$\sigma(\cL)=\sigma(M_\f)$$
    and, by the spectral mapping theorem,
    \[
    \sigma(M_\f)=\f(\sigma(M_\xi)),
    \]
    where $M_\xi$ is the multiplication by the independent variable, namely, $M_\xi g(\xi)=\xi g(\xi)$, for $\xi\in \T$ and $g\in L^2(\T)$. Since  $\sigma(M_\xi)=[-1/2,1/2]$,  we obtain $\f(\sigma(M_\xi))=[0,2]$.

Now take $g\in L^2(\T)$ and compute
    \begin{align*}
        \|g\|_{2}^2&=\int_{-1/2}^{1/2} |g(\xi)|^2 d\xi=\int_0^{1/2}|g(\xi)|^2 d\xi+\int_{-1/2}^0 |g(\xi)|^2 d\xi\\
        &=  \int_0^{1/2}\left(|g(\xi)|^2 +|g(1-\xi)|^2\right) d\xi.
    \end{align*}
    We make the change of variable $z=\f(\xi)$, where $\f$ is defined in \eqref{fi}, so that
    \begin{equation*}
        \xi=\frac1{2\pi}\arccos(1-z),\quad d\xi=\frac1{2\pi}\frac{dz}{\sqrt{(1-(1-z)^2)}}=\frac1{2\pi}\frac{dz}{\sqrt{z(2-z)}},
    \end{equation*}
    and
    $$    \|g\|_{2}^2=\int_0^2\left( \left|g\left(\frac1{2\pi}\arccos(1-z)\right)\right|^2+\left|g\left(1-\frac1{2\pi}\arccos(1-z)\right)\right|^2\right)\frac1{2\pi}\frac{dz}{\sqrt{z(2-z)}} =\|Wg\|^2_X$$
    Hence $W$ is isometric, $$    \|W g\|_{X} = \|g\|_{2},$$
    where
    \begin{equation*}\label{tilde u-bis}
       Wg(z)=\left(g\left(\frac1{2\pi}\arccos(1-z)\right),g\left(1-\frac1{2\pi}\arccos(1-z)
        \right)\right).
    \end{equation*}
    
  The range of $W$ is the whole space $X$; indeed, given $(u_1,u_2)\in X$, the function $g$ defined by 
$$g(\xi)=\left\{\begin{array}{ll}
        u_1(\varphi(\xi)), &\xi\in[0,1/2), \\
       u_2(\varphi(\xi)), & \xi\in[-1/2,0),
   \end{array}\right.$$
    and extended $1$-periodically, belongs to $L^2(\T)$ and satisfies $Wg=(u_1,u_2).$ In particular, by periodicity, for $\xi=\frac1{2\pi}\arccos(1-z)\in[0,\frac12),$ 
    $$g(1-\xi)=g(-\xi)= u_2(\varphi(-\xi))=u_2(\varphi(\xi))=u_2(z).$$
 Hence, $W$ is unitary.
 The  operator $\cU$  given by
    $\cU=W\cF, \quad \cU:\ell^2(\bZ)\to X$ is then unitary.

Furthermore, it is easy to see that $M_{\f}=W^{-1} M_z W$ and since $\cL =\cF^{-1} M_\f \cF,$ we have that
$\cU^{-1} M_z\,\cU = \cL$.
       \end{proof}

    To summarize we have the following commutative diagram:
            \[
        \begin{array}{cccccc}
        &\ell^2(\bZ)  \xrightarrow{\mathcal{F}} & L^2(\T)   \xrightarrow{W}   & X  \\
        & \downarrow{\cL} & \downarrow{M_{\f}}    &  \downarrow{M_z} \\
        & \ell^2(\bZ) \xrightarrow{\mathcal{F}} & L^2(\T) \xrightarrow{W} & X
    \end{array}
    \]
\begin{remark}
Note that the operator $\cU$ gives the decomposition of the spectral theorem for $\cL$.
The multiplicity function is equal to 2  $\mu$-a.e  in $[0,2]$ and $X$ is
a direct integral over $\sigma(\cL)$:
\[
X = \int_{\sigma(\cL)}^{\oplus} H_z \,d\mu(z), \quad\text{with } H_z = \C^2,\ \mu\text{-a.e.}\ z \in [0,2],
\]
and $\mu$ is the scalar spectral measure in $[0,2]$ (see for example \cite{Kriete86,AK73}).
\end{remark}

\begin{remark}  
Given $\omega\in [0,2]$, by Definition \ref{DefPW}, the Graph Paley-Wiener space $GPW_\omega(\bZ)$  is
$$GPW_\omega(\bZ)=\{f\in\ell^2(\bZ):\, \cU f(z)=0, \mu \text{-a.e}.\, z\in [\omega,2]\}.$$
\end{remark}
 It is clear that
 $f\in GPW_\omega(\bZ)$ if and only if $ \cF(f)(\xi)=0,$ where
$ \f(\xi)>\omega,$
hence
$$\cF(GPW_\omega(\bZ))=\{g\in L^2(\T): g(\xi) = 0 \,\ \mbox{if}\,\ \xi\in \T\setminus{\Omega}\}:=L^2(\Omega),$$
where $\Omega=\f^{-1}([0,\omega])$, and $\f$ is defined in \eqref{fi}.
We emphasize that functions in $L^2(\Omega)$ are $1$-periodic with support in $\Omega\subset\T.$

By Theorem \ref{teospectr} we have $\cU(GPW_\omega(\bZ))\subseteq
L^2([0,2], \bC^2;\mu)$.
\medskip

\begin{remark}
The Graph Paley--Wiener space is well defined for any
$\omega \in (0,2]$, or equivalently, for any symmetric interval in the group
$\T$. This construction yields a closed subspace of
$\ell^2(\Z)$. However, the Classical Sampling Theorem in
$\ell^2(\Z)$ makes sense only when the corresponding interval in
$\T$ has length $1/\ell$ for some positive integer $\ell$.

The reason for this is that we will  evaluate the functions in the
Paley--Wiener space on a sub-lattice of $\Z$, namely $\ell\Z$.
Consequently, the interval defining the Paley--Wiener space must
be a fundamental domain of the quotient group $\T$ with the
dual lattice $\{k/\ell\}_{k \in \Z_{\ell}}$. See \cite{K1965}.

This obstacle does not pose a problem in applications if we are willing
to increase the sampling rate.
Let us denote
\[\label{wl}
\omega_\ell = \varphi(1/2\ell), \qquad \ell \in \N.
\]
Then every function $f \in \ell^2(\Z)$ belongs to some
Graph Paley--Wiener space $GPW_{\omega}$ for a certain
$\omega \in (0,2]$.
Since there exists an  $\ell \in \N$ such that $\omega_{\ell+1} < \omega \leq \omega_{\ell}$ and
$GPW_{\omega} \subseteq GPW_{\omega_{\ell}}$,
it follows that $f$ can be sampled
on the sub-lattice $\ell\Z.$\end{remark}

Thus, we need to put some restriction in the values of $\omega$ in order to establish the CST in this context.
For convenience we will choose symmetric intervals around the origin in $\T.$
Observe that $\Omega=\f^{-1}([0,\omega])$ is symmetric with respect to the origin and conversely,  due to the
symmetry of $\f$, any symmetric subinterval of $\Omega$ is of the type $\f^{-1}([0,\omega])$ for some $\omega.$

\begin{figure}[h!]
    \centering
    \includegraphics[width=0.6\textwidth]{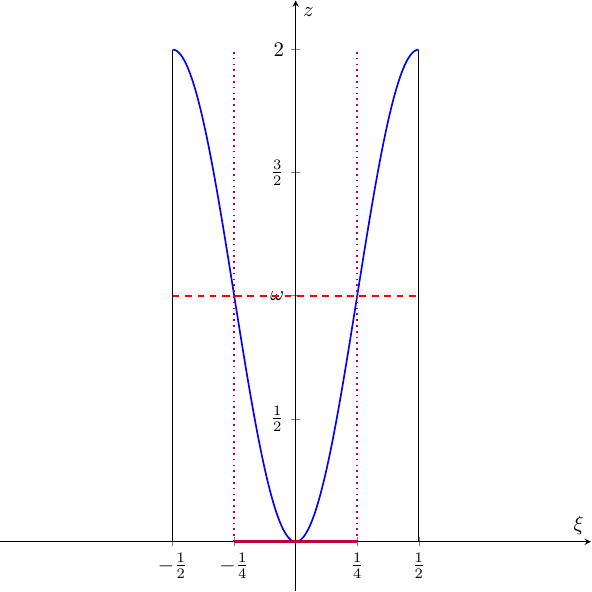}
    \caption{Graph of $\varphi$ in blue. Interval $\Omega =\varphi^{-1}([0,\omega])$ in red.}     \label{fig:cosine_region}
\end{figure}

\subsection{The Classical Sampling Theorem in \texorpdfstring{$GPW_{\omega}(\bZ)$}{GPW\_\{\textbackslash omega\}(Z)}}
We will now exhibit  an adaptation of the Classical Sampling Theorem to our graph setting (see Theorem \ref{ShannonT}).

Let $\omega \in [0,2]$  such that $\omega=\varphi(1/2\ell)$  for some $\ell\in \bN$ and consider the interval
\begin{equation}
\Omega = [ \frac{-1}{2\ell},\frac{1}{2\ell}] \subset \T.
\end{equation}

\begin{theorem}[Classical Sampling Theorem in $GPW_{\omega}(\bZ)$]\label{ShannotheoremZ}

Fix $\ell \in\bN$ and $\omega=\varphi(1/2\ell)$.

Let $\Omega =  [ \frac{-1}{2\ell},\frac{1}{2\ell}] \subset \T$.    Then for every $f \in GPW_{\omega}(\bZ)$ we have
\begin{equation}\label{ShannonZeq}
    f(m) = \sum_{j \in \mathbb{Z}} f(\ell j) \cdot \frac{\sin \pi (\frac{m}{\ell} - j)}{\pi (\frac{m}{\ell} - j)}=\sum_{j \in \mathbb{Z}} f(\ell j)\sinc(\frac{m}{\ell} - j),\quad m\in\bZ,
\end{equation}
    where the convergence is in $\ell^2(\bZ)$, and the sinc function is recalled in \eqref{eq:sinc}.
    \end{theorem}
This result shows that a sequence with {\em graph bandwidth} no greater than $\omega$, or equivalently with Fourier bandwidth $1/2\ell$, can be recovered by sampling on the lattice $\ell\Z$. It is a reformulation, in graph-spectral terms, of the classical sampling theorem on $\bZ$ (see also \cite{K1965} for the LCA-group version).
\begin{proof}
Let $\Omega=[-\frac{1}{2\ell},\frac{1}{2\ell}]=\f^{-1}([0,\omega]).$
    Observe that the set of $1$-periodic functions,
    \[
    \left\{ \frac{1}{\sqrt{|\Omega|}} e^{2\pi i \xi m / |\Omega|} \right\}_{m \in \mathbb{Z}}
    \]
    is an orthonormal basis of $L^2(\Omega)=\cF ({GPW}_{{\omega}}(\Z))\subset L^2(\T)$.

Let $g \in L^2(\Omega)$, and suppose $g= \cF{f},$ with $f\in {GPW}_{\omega}(\Z).$

Then,
\[
g(\xi) = \sum_{j \in \mathbb{Z}} \left\langle g, \frac{1}{\sqrt{|\Omega|}} e^{2\pi i j \xi / |\Omega|} \right\rangle_{L^2(\Omega)} \cdot \frac{1}{\sqrt{|\Omega|}} e^{2\pi i j \xi / |\Omega|}.
\]

That is,
\[
\cF{f}(\xi) = \sum_{j \in \mathbb{Z}} \left\langle\cF(f) , \frac{1}{\sqrt{|\Omega|}} e^{2\pi i j \xi / |\Omega|} \right\rangle_{L^2(\Omega)} \cdot \frac{1}{\sqrt{|\Omega|}} e^{2\pi i j \xi / |\Omega|}.
\]

Now, take inverse Fourier transform
\[
f(m) = \int_{\mathbb{T}} \cF{f}(\xi) e^{-2\pi i \xi m} \, d\xi
\]
and substitute the expression of $\cF{f}(\xi)$. Since $f\in {GPW}_{\omega}(\Z):$
\begin{align*}
f(m)& = \sum_{j \in \mathbb{Z}} \left\langle g, \frac{1}{\sqrt{|\Omega|}} e^{2\pi i j \xi / |\Omega|} \right\rangle_{L^2(\Omega)} \cdot \frac{1}{\sqrt{|\Omega|}} \int_{\mathbb{T}} \chi\strut_{\Omega}(\xi) e^{-2\pi i \xi (m - j/|\Omega|)} \, d\xi\\
&=\frac{1}{\sqrt{|\Omega|}} \sum_{j \in \mathbb{Z}} \left\langle g, \frac{1}{\sqrt{|\Omega|}} e^{2\pi i j \xi / |\Omega|}\right\rangle_{L^2(\Omega)} \cdot \cF^{-1}(\chi\strut_{\Omega}) \left(m - \frac{j}{|\Omega|} \right).
\end{align*}

Now, observe that

\begin{eqnarray*}
\left\langle g, \frac{1}{\sqrt{|\Omega|}} e^{2\pi i j \xi / |\Omega|} \right\rangle_{L^2(\Omega)}
 \!\!\!\!\!\!\!\!&=&\!\!\!\! \int_{\Omega} g(\xi) \cdot \frac{1}{\sqrt{|\Omega|}} e^{-2\pi i j \xi / |\Omega|} \, d\xi = \frac{1}{\sqrt{|\Omega|}} \int_{\Omega} \cF(f)(\xi) e^{-2\pi i j \xi / |\Omega|} \, d\xi\\ \\
 &= &\!\!\!\!\sqrt{\ell} f(\ell j),
\end{eqnarray*}

Hence,
\begin{equation}\label{fm}
f(m) = \ell\sum_{j \in \mathbb{Z}} f(\ell j) \cdot \cF^{-1}(\chi_{\ds \Omega})(m -  \ell j).
\end{equation}

Also, since $|\Omega|=1/\ell,$  by periodicity and the hypothesis on the symmetry of $\Omega,$
\begin{equation}\label{chiinv2}\ell \cF^{-1}(\chi_{\ds \Omega})(m -  \ell j)=\ell\int_{\mathbb{T}} \chi_{\ds \Omega}(\xi) e^{-2\pi i \xi(m-\ell j)} \, d\xi =
\ell\int_{-\frac{1}{2\ell}}^{\frac{1}{2\ell}}  e^{-2\pi i \xi(m-\ell j)} \, d\xi
=\sinc(\frac{m}{\ell}-j).
\end{equation}


 Replacing the previous equality in \eqref{fm} we obtain \eqref{ShannonZeq}.
\end{proof}

Theorem \ref{ShannotheoremZ} tells us that any sequence
$f$ in the Graph Paley--Wiener space $GPW_{\omega}(\Z)$,
with $\omega = \f(1/2\ell)$, can be perfectly reconstructed from its
values on the lattice $\ell\Z$, that is, from $\{f(\ell j)\}_{j \in \Z}$.
The theorem rephrases the Classical Sampling Theorem in the language of graph signals.

\subsection{Dynamical frames on the graph.}

For each $j\in \bZ,$ denote by $g_j$ the sequence $g_j(m) =\frac{1}{\sqrt{\ell}}\,\text{sinc}(m/\ell - j).$

Thus,  $\sqrt{\ell}\,f(\ell j) = \,< f, g_j >_{\ell^2(\bZ)}$ and $\{g_j\}_{j\in\bZ}$  is an orthonormal basis of $GPW_{\omega}(\Z).$
(see the Lemma in Appendix).

The sequence $\{g_j\}_{j\in\mathbb{Z}}$ forms, in particular, a frame for $GPW_{\omega}(\Z)$.
We may also choose any other frame $\{q_k\}_{k\in J}$ for $GPW_{\omega}(\Z)$ and use it to reconstruct any $f\in GPW_{\omega}(\Z)$ from its measurements $\{\langle f, q_k\rangle\}_{k\in J}$.

Using the frame $\{g_j\}_{j\in\mathbb{Z}}$ has the advantage that the coefficients of $f$ correspond directly to pointwise evaluations on a coarser lattice, which can sometimes be beneficial.
However, in practical applications it is often unrealistic to assume access to exact pointwise values due to inevitable measurement errors.

Moreover, while an orthonormal basis offers no redundancy, redundant frames can provide greater robustness to noise and data loss, making them more suitable in many real-world scenarios.

There are several ways to obtain frames in $GPW_{\omega}(\Z)$. For example the orthogonal projection of an  orthonormal basis of $\ell^2(\bZ)$ onto 
the Paley-Wiener space  $GPW_{\omega}(\Z)$ is a Parseval frame of $GPW_{\omega}(\Z).$

In what follows we will show how to construct {\it frames of iterations} of $GPW_{\omega}(\Z)$, using results from dynamical sampling.

Let $\omega\in(0,2]$ and $\Omega=\f^{-1}([0,\omega])=[-\frac{1}{2\ell},\frac{1}{2\ell}]$, where now $\ell:=1/|\Omega|>0$ is any positive real number; the case
\begin{equation}\label{l1}
\ell\in\N,\qquad \omega=\f\Big(\frac{1}{2\ell}\Big),\qquad\Omega=[-\frac{1}{2\ell},\frac{1}{2\ell}],
\end{equation}
is the one in which the frame coefficients can be read as lattice samples (Theorem~\ref{main2}).
We will use Theorem \ref{TeorCarl} to construct a frame of iterations in $L^2(\Omega)$ and then using the inverse of the Fourier transform
we will obtain a frame of the Paley-Wiener space $GPW_{\omega}(\bZ).$ To apply Theorem \ref{TeorCarl} we first need to
have an orthonormal basis of $L^2(\Omega)$; for any interval $\Omega$ the functions in \eqref{bon} below form such a basis, so no restriction on $\omega$ is needed at this stage.

\begin{theorem}\label{main}
    Consider $\omega\in (0,2]$ and $\Omega=\f^{-1}([0,\omega])$.  Let $\Lambda=\{\lambda_j\}_{j\in\N}$ be an uniformly separated sequence in the open unit disc $\{z\in\C: |z|< 1\},$ and consider   the orthonormal basis $(e_j)_{j\geq 0}$ in $L^2(\Omega)$ given by
    \begin{equation}\label{bon}
    e_j(\xi)=\frac{1}{\sqrt{|\Omega|}}{\chi}_{\ds \Omega}(\xi)\, e^{2\pi i \frac{s(j)}{|\Omega|}\xi}, \quad j\geq0, \quad\xi\in \Omega,
    \end{equation}

where $s:\bN\to\bZ$ is any fixed bijection.

  Define the operator $T,$   by
\begin{equation}
    T f = \sum_{j=0}^\infty \lambda_{j} < f, e_j > e_j, \; \;\; f\in L^2(\Omega).
\end{equation}
and a  fixed function
 $a\in L^2(\Omega),$ such that
 \begin{equation}\label{a}
 |\la a,e_j\ra|^2\sim 1-|\lambda_j|^2.
\end{equation}

Then the system $\{T^n a: n\in \N\}$ is a frame of $L^2(\Omega).$
\end{theorem}
\begin{proof}
Since $|\lambda_j| < 1$ for all $j \geq 0$ then $T$ defines a bounded operator in  $L^2(\Omega)$.
The proof follows applying  Theorem \ref{TeorCarl} to $H=L^2(\Omega),$ the sequence
$\Lambda,$ the operator $T$ and the function $a.$
\end{proof}

\begin{remark}\label{after}
Note that \eqref{a} is equivalent to
$$\ell\,\big|\,\widetilde{f_0}\left(\tfrac{s(j)}{|\Omega|}\right)\big|^2\sim 1-|\lambda_j|^2,$$ where $f_0=\cF^{-1}a$ and $\widetilde{f_0}(x)=\int_\Omega a(\xi)e^{-2\pi i x\xi}\,d\xi$, $x\in\R$, is its bandlimited extension. In particular, $\widetilde{f_0}(s(j)/|\Omega|)=f_0(\ell s(j))$ when $\ell\in\N$.
\end{remark}

Now using that the Fourier transform is a unitary operator that maps $GPW_\omega(\Z)$ onto $L^2(\Omega)$ we construct a frame of iterations
in the Paley-Wiener space:
\begin{theorem}\label{main2}
Let  $T$ and $a \in L^2(\Omega)$ as in Theorem \ref{main} and set $R=\cF^{-1}\,T\,\cF.$
Let $\ell\in\N,$ such that $\f(1/2\ell)=\omega.$
Then, $R\in\cB(\ell^2(\bZ))$, satisfies $R(GPW_\omega(\bZ))\subset GPW_\omega(\bZ)$ and has the form
\begin{equation}\label{eq:R}
    (Rf)(k)=\sum_{j=0}^\infty \lambda_{j}f(\ell{s(j)})\sinc (k/\ell -s(j)),\quad k\in\bZ,
\end{equation}
 with convergence in $\ell^2(\bZ)$.  Setting $f_0=\cF^{-1}a$, the sequence $\{R^n\,f_0\}_{n\geq 0}$
is a frame for $GPW_\omega(\bZ)$.
\end{theorem}
\begin{proof}
By assumption, the following commutative diagram holds:
\begin{equation}\label{comm}
        \begin{array}{ccc}
        GPW_\omega(\Z) & \xrightarrow{\cF} & L^2(\Omega)\\
        \downarrow^{R} &  & \downarrow^{T} \\
        GPW_\omega(\Z) & \xrightarrow{\cF} & L^2(\Omega). \\
    \end{array}
\end{equation}
    For any $f\in GPW_\omega(\Z)$, we have $\cF{f}\in L^2(\Omega)$. Using \eqref{bon}, the fact that $\cF f$ has support in $\Omega$,  and the same pattern as in the proof of the  sampling  Theorem \ref{ShannotheoremZ},
    \begin{align*}
        T\,\cF{f}(\xi) &= \sum_{j=0}^\infty \lambda_{j} < \cF{f}, e_j> e_j(\xi)\\
        &= \ell\sum_{j=0}^\infty \lambda_{j}f(\ell{s(j)})\chi_{\ds \Omega}(\xi)\, e^{2\pi i \ell s(j)\xi},
    \end{align*}
    where $\ell=1/|\Omega|$ as  in \eqref{l1}. Further, we use the commutative diagram \eqref{comm} so that $R=\cF^{-1}\,T\,\cF$. Hence, for every $k\in\bZ$,

    \begin{align*}
    Rf(k)&= \ell\sum_{j=0}^\infty \lambda_{j}f(\ell{s(j)})
    \int_{\mathbb{T}} \chi_{\ds \Omega}(\xi)\,e^{2\pi i \ell s(j)\xi} e^{-2\pi i k\xi}d\xi\\
    &=\ell \sum_{j=0}^\infty \lambda_{j}f(\ell {s(j)})\int_{-\frac{1}{2\ell}}^{\frac{1}{2\ell}}
    \, e^{-2\pi i(k- \ell s(j))\xi}d\xi\\
    &=\sum_{j=0}^\infty \lambda_{j}f(\ell {s(j)})\frac{\sin \pi (k/\ell - s(j))}{\pi (k/\ell -s(j))}\\
        &=\sum_{j=0}^\infty \lambda_{j}f(\ell {s(j)})\sinc  (k/\ell - s(j)).
\end{align*}


This follows by \eqref{chiinv2}.

Hence,  \eqref{eq:R} is proved.
Now, since
$$R^nf_0=\cF^{-1}T^n \cF f_0=\cF^{-1}T^n a$$
and the sequence $(T^n a)_{n\geq 0}$,  is a frame for $L^2(\Omega)$,
we have that $(R^nf_0)_{n\geq 0}$ is a frame for $GPW_\omega(\bZ).$ \end{proof}

Now, since $(R^nf_0)_{n\geq 0}$
    is a frame for $GPW_\omega(\bZ)$, then any signal $f\in GPW_{\omega}(\bZ)$ can be recovered from the measurements
    $\{\la f, R^nf_0\ra\}_{n\geq 0}$ using any dual frame $\{g_n\}_n$ of $\{R^nf_0\}_{n\geq 0}$ in $GPW_{\omega}(\bZ),$ that is,

\begin{equation}\label{RC}
    f=\sum_{n\geq0}\la f, R^nf_0\ra \;g_n.
\end{equation}
The frame coefficients  $\la f, R^nf_0\ra$ are   of the form
\begin{align}\label{coefR}
    \la f, R^nf_0\ra&=\ell\,\sum_{j\geq0}f(\ell s(j))\overline{\lambda^n_jf_0(\ell s(j))}.
\end{align}

Indeed, since
 $R^nf_0=\cF^{-1}T^n\cF f_0=\cF^{-1}T^n a,$
    and using Theorem \ref{main2} we have
    $$    (R^nf)(k)=\sum_{j=0}^\infty \lambda^n_{j}f(\ell s(j))\sinc(k/\ell -s(j)),\quad k\in\bZ.$$

Now, consider $f_0=\cF^{-1}a$, with $a$ in \eqref{a}. For every $f\in GPW_{\omega}(\bZ)$,

\begin{align*}
    \la f, R^nf_0\ra&=\sum_{k\in\bZ}\sum_{j\geq0}
f(k)\overline{\lambda^n_j\,f_0(\ell {s(j)})}\sinc(k/\ell -s(j))\\
&=\sum_{j\geq0}\overline{\lambda^n_j\,f_0(\ell {s(j)})}
\sum_{k\in\bZ}f(k)\sinc(k/\ell -s(j))\\
&=\ell\,\sum_{j\geq0}\overline{\lambda^n_j\,f_0(\ell {s(j)})} f(\ell s(j)),
\end{align*}
where we used \eqref{repker} in the last equality.

Observe that we have just shown that every function $f\in GPW_{\omega}(\bZ)$ can be recovered from the
samples
\[
\{\langle f,R^n f_0\rangle\}_{n\geq 0}.
\]
Since
\[
\langle f,R^n f_0\rangle
=
\langle (R^*)^n f,f_0\rangle,
\]
this reconstruction admits a natural interpretation in the framework of
dynamical sampling. Indeed, one may view $f$ as evolving in time under the
action of the evolution operator $R^*$, while at each time $n$ the evolved
state $(R^*)^n f$ is sampled using the same sensor $f_0$.

\subsection{The $n$-dimensional case \texorpdfstring{$\bZ^n$}{Zn}}\label{subsec:zn}
A comparable result applies to lattices $\mathbb{Z}^n$ in arbitrary dimension.  In this case, the Fourier transform $\mathcal{F}$ on the Hilbert space $\ell^2(\mathbb{Z}^n)$ is defined as a unitary operator given by the expression
\[
(\mathcal{F}f)(\xi_1,\dots,\xi_n) = \sum_{(k_1,\dots,k_n) \in \mathbb{Z}^n} f(k_1,\dots,k_n) e^{ 2\pi i(k_1\xi_1 +\dots+ k_n\xi_n)}, \quad f \in \ell^2(\mathbb{Z}^n),
\]
where $(\xi_1,\dots, \xi_n)$ ranges over the torus $\mathbb{T}^n := [-\frac{1}{2}, \frac{1}{2})^n$. This operator $\mathcal{\cF}$ constitutes an isometric isomorphism from $\ell^2(\mathbb{Z}^n)$ to $L^2(\mathbb{T}^n, d\xi_1,\dots d\xi_n)$, where $\mathbb{T}^n$ denotes the $n$-dimensional torus.

For the normalized discrete Laplacian on $\mathbb{Z}^n$ the following identity holds:
\[
\cF(\cL f)(\xi_1,\dots,\xi_n) = \frac{2}{n}\left( \sin^2(\pi\xi_1) +\dots+ \sin^2(\pi\xi_n) \right) \cF (f)(\xi_1,\dots,\xi_n).
\]
The graph $\mathbb{Z}^n$ has constant degree $d(\textbf{k})=2 n,$ where we write
$\bk=(k_1,\dots,k_n)\in\bZ^n$ for short. Hence
$$ \cL f(\bk)=\frac{1}{2n}\left(\sum_{{\bh}\sim{\bk}} f(\bk)-f(\bh)\right)
=\frac{1}{2n}\left(2n f(\bk)-\sum_{j=1}^n f(\bk-\delta_j)-\sum_{j=1}^n f(\bk+\delta_j)\right),$$
where $\delta_j\in\bZ^n$ has $1$ in the $jth$ position and zeroes elsewhere.
Therefore
\begin{align*}
\mathcal{F}(\mathcal{L} f)({\boldsymbol{\xi}})&=
\frac{1}{2n}\left(2n \mathcal{F}(f)({\boldsymbol{\xi}})-\sum_{j=1}^n e^{-2\pi i\delta_j\cdot \boldsymbol{\xi}} \mathcal{F}(f)(\boldsymbol{\xi})-
\sum_{j=1}^n e^{2\pi i\delta_j\cdot \boldsymbol{\xi}} \mathcal{F}(f)(\boldsymbol{\xi})\right)\\ \\
&=
 \mathcal{F}(f)(\boldsymbol{\xi})\left(1-\frac{1}{2n}\sum_{j=1}^n (e^{-2\pi i\delta_j\cdot \boldsymbol{\xi}} +e^{2\pi i\delta_j\cdot \boldsymbol{\xi}})\right)\\ \\
 &=
 \frac{1}{n}\mathcal{F}(f)(\boldsymbol{\xi})\left(n-\sum_{j=1}^n \frac{e^{-2\pi i\delta_j\cdot \boldsymbol{\xi}} +e^{2\pi i\delta_j\cdot \boldsymbol{\xi}}}{2}\right)
 =
 \frac{1}{n}\mathcal{F}(f)(\boldsymbol{\xi})\left(n-\sum_{j=1}^n \frac{e^{-2\pi i\xi_j} +e^{2\pi i\xi_j}}{2}\right)\\ \\
 &=\frac{1}{n}\mathcal{F}(f)(\boldsymbol{\xi})\left(n-\sum_{j=1}^n \cos(2\pi \xi_j)\right)=\mathcal{F}(f)(\boldsymbol{\xi})\left(\frac{2}{n}\sum_{j=1}^n \sin^2(\pi \xi_j)\right)
\end{align*}

This leads to the following theorem, cf. Theorem 5.4 of \cite{PEsensonTAMS2008}:
\begin{theorem}
    The spectrum of the Laplace operator on the lattice $\mathbb{Z}^n$ is exactly the interval $[0, 2]$.
\end{theorem}
\begin{figure}
  \centering
  \includegraphics[width=10cm]{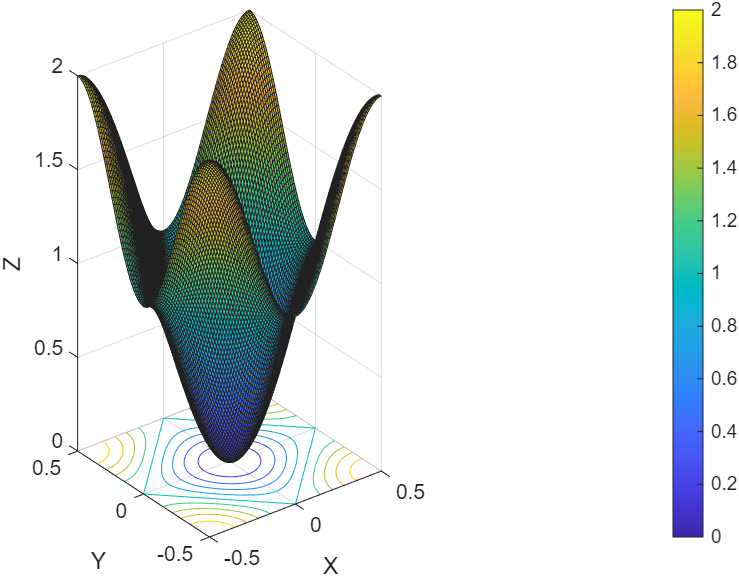}\\
  \caption{Boundary  of $\Omega_2$ $n=2,$ for different values of $\omega.$}\label{level4}
\end{figure}

A function $f \in \ell^2(\mathbb{Z}^n)$ belongs to the Paley–Wiener space $GPW_\omega(\mathbb{Z}^n)$ for $0 \leq \omega \leq 2$ if and only if the support of its Fourier transform $\mathcal{\cF}f$ is contained in the measurable set $ \Omega_n\subset[-\frac{1}{2},\frac{1}{2})^n$ defined by\begin{equation}\label{omegan}
 \Omega_n=\{ (\xi_1,\dots,\xi_n)\in\T^n: \frac{2}{n}(\sin^2(\pi \xi_1) +\dots+ \sin^2(\pi \xi_n) )\leq \omega \}.
 \end{equation}

The proof follows analogously to the one-dimensional case.
From now on we shall use the notation $\Omega_n$ for the above set.

Thus, by Definition \ref{DefPW}, for each $\omega \in [0,2]$ we have
\begin{align}\label{GPW}
GPW_\omega(\mathbb{Z}^n)&= \{f\in\ell^2(\bZ^n) : \text{supp}(\mathcal{F}f )\subseteq \Omega_n\},
\end{align}
and $\cF\cL\cF^{-1}$ is the multiplication operator by the symbol $\frac2n\sum_j\sin^2(\pi\xi_j)$.

The idea, once again, is to construct an operator whose iterates generate a frame for $L^2(\Omega_n)$. The main difficulty is that, for each value of $\omega$, the domain $\Omega_n$ may have a complicated geometry, making it difficult to apply the method developed for the group $\mathbb{Z}$. Our strategy is therefore to replace $\Omega_n$ with an equivalent domain contained in $[-1/2,1/2]^n$ that is more amenable to analysis. This construction is described in the following section.
\subsubsection{General method in the $n$-dimensional case}
Suppose $\Omega_n\subset\T^n$ is mapped by a (piecewise) $\mathcal{C}^1$, bijective map to a domain $D\subset\T^n,$
and that we know
an orthonormal basis of $L^2(D)$ ($1$-periodic functions in all directions with support in $D$). Call this orthonormal basis $\{\psi_n\}_{n\geq 0}.$


Denote this $\mathcal{C}^1$ bijection by $P:\T^n\rightarrow \T^n,$ and assume that its Jacobian $J_P$ is 
not zero on
$\Omega_n,$ with obvious modifications in the case of $P$ piecewise $\mathcal{C}^1$ (this is the case for the piecewise translations used below, where $|J_P|\equiv1$). Then both  maps
\begin{eqnarray}
& &S:L^2(\Omega_n)\rightarrow L^2(D),\quad Sg=g\circ P_{|D}^{-1}:D\rightarrow \C,\nonumber\\
& &\label{sbije}\\
& & S^{-1}:L^2(D)\rightarrow L^2(\Omega_n),\quad S^{-1}f=f\circ P_{|\Omega_n}:\Omega_n\rightarrow \C,\nonumber
\end{eqnarray}
preserve periodicity.

After a change of variable
(possibly in each subdomain where $P$ is $\mathcal{C}^1$),
$$\int_D |Sf(\tau)|^2\, d\tau=\int_D |f(P^{-1}(\tau))|^2\, d\tau=\int_{\Omega_n}|f(\xi)|^2\,|J_P(\xi)|\,d\xi,$$
so that, since $J_P$ is bounded above and below on $\Omega_n$, $S$ is a bounded, invertible operator with bounded inverse.

We recall the following general result (see \cite[Corollary~5.3.2]{Olebook})
\begin{theorem}\rev{\label{thm:framesinvertible}}
Let $\mathcal{H}, \mathcal{K}$ be separable Hilbert spaces, $F: \mathcal{H}\rightarrow \mathcal{K}$
be a bounded and invertible operator. Then $F$ sends frames of $\mathcal{H}$ into frames of
$\mathcal{K}.$
\end{theorem}
Now, starting from the orthonormal basis $\{\psi_n\}_{n\geq 0}$ in $L^2(D),$ and a uniformly separated sequence $(\lambda_j)_{j\geq 0},$ we consider the operator $T$ as in \eqref{diag}
\begin{equation}\label{ndimT}
T:L^2(D)\rightarrow L^2(D), \quad T=\sum_{j=0}^{+\infty} \lambda_j \l \cdot,\psi_j\r \psi_j
\end{equation}
and $h_0\in L^2(D)$ whose coefficients verify condition \eqref{coeff}. By Theorem \ref{TeorCarl}
we get a frame of iterations $(T^n h_0)_{n\in\N}$ for $L^2(D).$

It follows that, setting $f_0=\cF^{-1}\,S^{-1} h_0\in \ell^2(\Z^n),$ and
$$R=\cF^{-1}\,S^{-1} T S \cF :\ell^2(\Z^n)\rightarrow \ell^2(\Z^n),$$
we get $R(GPW_{\omega}(\Z^n))\subset GPW_{\omega}(\Z^n),$ and the sequence
$(R^n f_0)_{n\in\N}$ is a frame of iterations  for $GPW_\omega(\Z^n).$

Now, as one can see from Figure \ref{level4}, for a fixed $n$, one cannot expect to use the same
bijection $P$ for all values of $\omega\in[0,2]$. In the following we treat the case $n=2$, $\omega=1$.

\subsubsection{The case $n=2$ and $\omega=1$}
In the case $n=2$ and $\omega=1,$  $\xi=(\xi_1,\xi_2)\in\Omega_2$ verifies
$$\sin^2({\pi}\xi_1)+\sin^2({\pi}\xi_2)\leq 1\Leftrightarrow \cos(2\pi\xi_1)+\cos(2\pi\xi_2)\ge0\Leftrightarrow |\xi_1|+|\xi_2|\le\frac12,$$
so that $\Omega_2$ is the diamond of area $1/2$ with vertices $(\pm\frac12,0)$, $(0,\pm\frac12)$ (the right-hand equivalence follows from $\cos a+\cos b=2\cos\frac{a+b}{2}\cos\frac{a-b}{2}$).
Consider the piecewise translation
\begin{equation}\label{bij21}
P:\T^2\rightarrow \T^2,
\end{equation}
which sends $\Omega_2\subset \T^2$ onto the rectangle
$D=\left[-\frac{1}{2},0\right]\times\left[-\frac{1}{2},\frac{1}{2}\right]\subset\T^2,$
by leaving unchanged the left half of $\Omega_2,$ i.e. the set
$$D_1=\{ -\frac{1}{2}< \xi_1< 0,\; -\xi_1-\frac{1}{2}< \xi_2< \xi_1+\frac{1}{2}\},$$
and exchanging  by translation
$$D_2=\{ 0\leq \xi_1\leq \frac{1}{2},\; 0\leq \xi_2\leq-\xi_1+\frac{1}{2}\}$$ with
$$\widetilde{D}_2=D_2+(-\frac{1}{2},-\frac{1}{2})=\{ -\frac{1}{2}\leq \xi_1\leq0,\; -\frac{1}{2}\leq \xi_2\leq-\xi_1-\frac{1}{2}\},$$
and
$$D_3=\{ 0\leq \xi_1\leq \frac{1}{2},\; \xi_1-\frac{1}{2}\leq \xi_2< 0\}$$ with
$$\widetilde{D}_3=D_3+(-\frac{1}{2},\frac{1}{2})=\{ -\frac{1}{2}\leq \xi_1\leq0,\; \xi_1+\frac{1}{2}\leq \xi_2< \frac{1}{2}\},$$
(translations are understood modulo $\bZ^2$).
\begin{figure}[t]
    \begin{minipage}[th]{.45\textwidth}
        \centering
        \includegraphics[width=\textwidth]{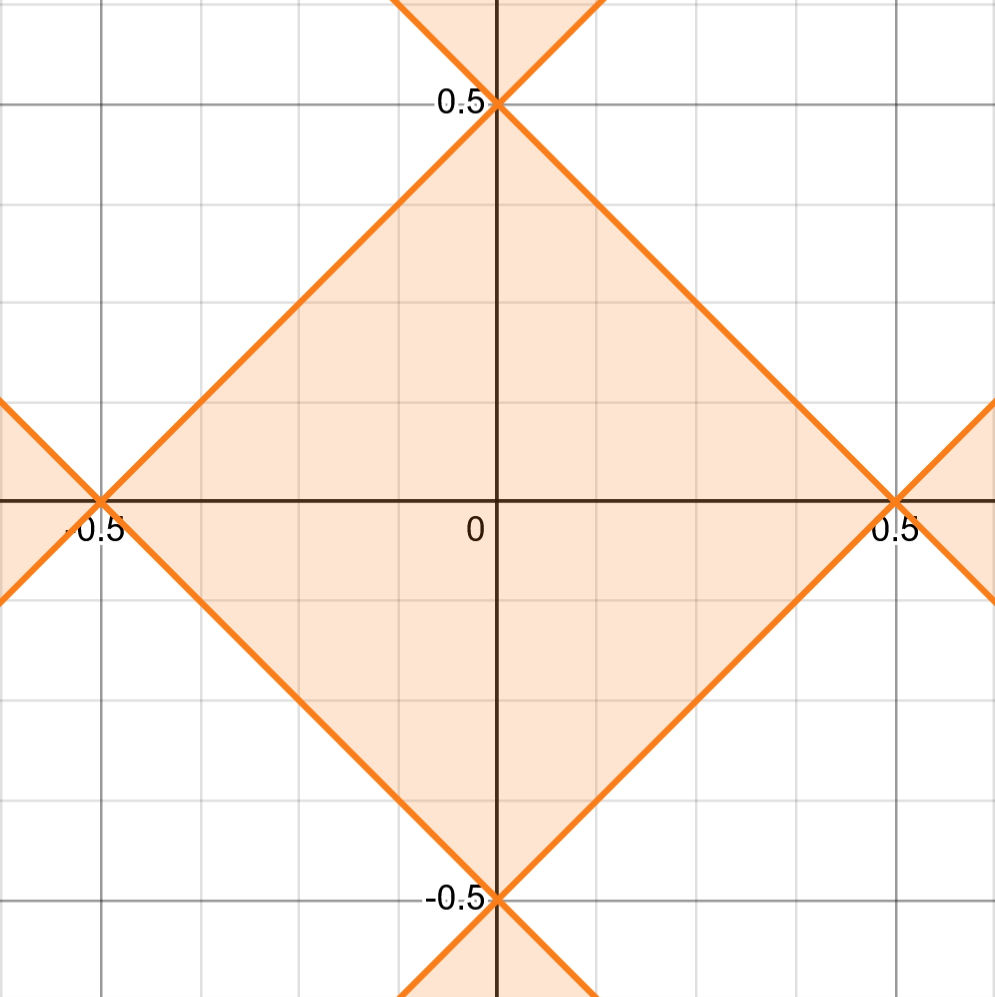}
       \caption{\!\!The domain $\Omega_2$.}\label{figOmega2}
    \end{minipage}
    \hfill
    \begin{minipage}[th]{.45\textwidth}
        \centering
        \includegraphics[width=\textwidth]{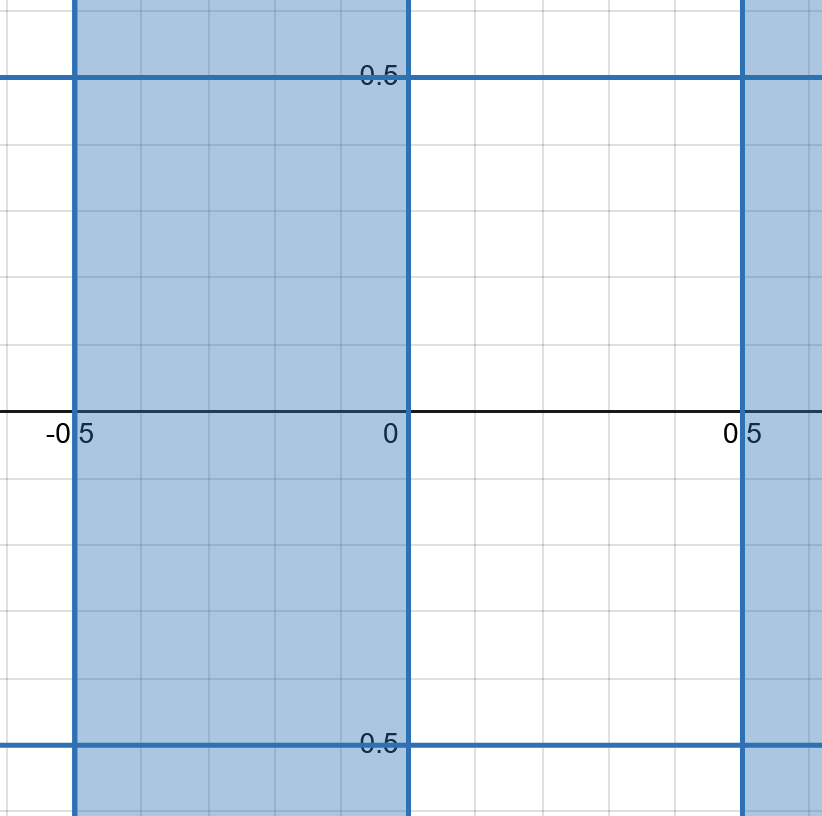}
       \caption{\!\!The image $P(\Omega_2)$.}\label{figOmega2Tr}
   \end{minipage}
\end{figure}

Therefore we can write in $\T^2,$ identified with $[-\frac{1}{2},\frac{1}{2})^2,$ 
\begin{equation}\label{21p}
P(\xi_1,\xi_2)=\left\{\begin{array}{ll}
(\xi_1-\frac{1}{2},\xi_2+\frac{1}{2}),&\text{if}\quad 0\leq \xi_1\leq \frac{1}{2},\;\xi_1-\frac{1}{2}\leq\xi_2< 0,\\
(\xi_1-\frac{1}{2},\xi_2-\frac{1}{2}),&\text{if}\quad 0\leq \xi_1\leq \frac{1}{2},\;0\leq\xi_2\leq -\xi_1+\frac{1}{2},\\
(\xi_1+\frac{1}{2},\xi_2+\frac{1}{2}),&\text{if}\quad -\frac{1}{2}\leq \xi_1\leq 0 ,\;-\frac{1}{2}\leq\xi_2\leq -\xi_1-\frac{1}{2},\\
(\xi_1+\frac{1}{2},\xi_2-\frac{1}{2}),&\text{if}\quad -\frac{1}{2}\leq \xi_1\leq 0,\;\xi_1+\frac{1}{2}\leq\xi_2< +\frac{1}{2},\\
(\xi_1,\xi_2),&\text{elsewhere}.
\end{array}\right.
\end{equation}

Define $S$ as in \eqref{sbije} and observe that $S$ is unitary, since $P$ is a piecewise translation. Hence in this case it is easier to find an orthonormal basis in $L^2(\Omega_2).$ Since $D$ has sides of length $1/2$ and $1$, an orthonormal basis
of $L^2(D)$ is
$$\{\sqrt{2}\,e^{2\pi i (2k_1,k_2)\cdot \tau}\Chi_{D},\, (k_1,k_2)\in\Z^2\}.$$
If we take  bijection $s:\N\rightarrow \Z^2,$ $s(n)=(s_1(n),s_2(n)),$ we obtain the basis of $L^2(D)$
\begin{equation}\label{ortho21}
\widetilde{\psi}_{n}(\tau)=\sqrt{2}\,e^{2\pi i (2s_1(n),s_2(n))\cdot \tau}\Chi_{D}(\tau),\quad n\in\N.
\end{equation}
The image of $\{\widetilde{\psi}_{n}\}$ by $S^{-1}$ is an orthonormal basis of $L^2(\Omega):$
\begin{eqnarray}
\psi_{n}(\xi)&=&S^{-1}(\widetilde{\psi}_{n})(\xi)=\widetilde{\psi}_{n}(P(\xi))=\sqrt{2}\,e^{2\pi i (2s_1(n),s_2(n))\cdot P(\xi)}\Chi_{D}(P(\xi))\nonumber\\ &&\label{ortho21Omega} \\
&=&
\sqrt{2}\,e^{2\pi i (2s_1(n),s_2(n))\cdot \xi}\left(
\Chi_{D_1}(\xi)+(-1)^{s_2(n)}\Chi_{D_2\cup D_3}(\xi)
\right)\nonumber
\end{eqnarray}
Consequently, with $T$ as in \eqref{ndimT}, $h_0\in L^2(D)$ satisfying \eqref{coeff}, $f_0=\cF^{-1}S^{-1}h_0$ and $R=\cF^{-1}S^{-1}TS\cF$, the sequence $(R^nf_0)_{n\in\N}$ is a frame of iterations for $GPW_1(\bZ^2)$, and $R=\sum_n\lambda_n\langle\cdot,\cF^{-1}\psi_n\rangle\cF^{-1}\psi_n$ on $GPW_1(\bZ^2)$.

\section{A finite-dimensional counterpart}\label{sec:finite}

Finite-dimensional dynamical sampling has been studied in a general
operator-theoretic setting, where the aim is to recover a vector from spatial
samples of its iterates under a linear operator; see, for example,
\cite{UCACHA2017,CabrelliMolterPaternostroPhilipp}. Related space--time sampling
problems on graphs have also been considered, particularly for bandlimited
graph signals evolving under diffusion-type dynamics and for randomized
sampling schemes; see
\cite{AldroubiBaileyKrishtalMillerPetrosyan2024,
HuangNeedellTang2024,HuangLiTangYao2025,YaoHuangTang2023}. Sampling,
filtering and interpolation on finite and infinite combinatorial graphs, in
the spirit of the graph Paley--Wiener spaces used throughout this paper, have
also been developed by \cite{PesensonPesenson2010,PesensonPesenson2021}.

In the present section, we formulate the finite-graph problem in the spectral
coordinates associated with the normalized Laplacian. For dynamics of the form
\[
R=h(\mathcal L),
\]
the samples of a signal in \(GPW_\omega(G)\) are interpreted as frame
coefficients of a finite family generated by iterates of the spectral
multiplier. Stable recovery is therefore characterized by the frame property
of this family, while the singular values of its analysis matrix determine the
optimal frame bounds and the conditioning of the corresponding frame operator.

We then use the path graph \(P_N\) as a numerical model to study how the
number of temporal iterates and the spatial distribution of the sampling
vertices affect the frame bounds and conditioning. The experiments compare
one- and two-endpoint sampling and include a spatially distributed
tight-frame benchmark.

We recall here that a Path graph $P_N$ is a sequence of $N$ vertices connected in a single line, where each vertex (except the two endpoints) is connected to exactly two others, i.e. its immediate neighbors, and there are no cycles or branches.

\subsection{Basic definitions and notation}

Let \(G=(V,E)\) be a finite, connected, undirected, weighted graph with
\(|V|=N\). We assume \(N\geq 2\) and identify \(V\) with
\(\{1,\ldots,N\}\). The graph is represented by a symmetric adjacency matrix
\(A=(a_{ij})_{i,j=1}^N\), where \(a_{ij}\geq 0\) and \(a_{ij}>0\) if and only if
vertices \(i\) and \(j\) are connected. We allow weights in this section; the unweighted case of the Introduction corresponds to $a_{ij}\in\{0,1\}$.

Let \(D\) be the diagonal degree matrix,
\[
D=\operatorname{diag}(d(1),\ldots,d(N)),
\qquad
d(i)=\sum_{j=1}^N a_{ij}.
\]
Since \(G\) is connected and \(N\geq2\), we have \(d(i)>0\) for every
\(i=1,\ldots,N\). The unnormalized graph Laplacian is $\widetilde{\mathcal L}=D-A.$
We work with the normalized Laplacian

\begin{equation}\label{normL}
\mathcal L
=
D^{-1/2}\widetilde{\mathcal L}D^{-1/2}
=
I-D^{-1/2}AD^{-1/2}.
\end{equation}

Since \(\mathcal L\) is real and symmetric, it admits an orthonormal spectral
decomposition $\mathcal L=U\Lambda U^*,$
where
\[
U=[\chi_0,\chi_1,\ldots,\chi_{N-1}]
\]
is unitary and
\[
\Lambda=\operatorname{diag}(\lambda_0,\lambda_1,\ldots,\lambda_{N-1})
\]
with
\[
0=\lambda_0<\lambda_1\leq \cdots \leq \lambda_{N-1}\leq 2.
\]
The strict inequality \(\lambda_0<\lambda_1\) follows from the connectedness of
\(G\). For the normalized Laplacian, the eigenvector associated with
\(\lambda_0=0\) is
\[
\chi_0(i)
=
\frac{\sqrt{d(i)}}{\sqrt{\operatorname{vol}(G)}},
\qquad
\operatorname{vol}(G)=\sum_{i=1}^N d(i).
\]
In particular, \(\chi_0\) is constant if and only if \(G\) is regular.

A graph signal is a function \(f:V\to\mathbb C\), identified with a vector
\(f\in\mathbb C^N\). The {\em graph Fourier transform} of \(f\) is defined by
\begin{equation}\label{gft}
\widehat f(\ell)
=
\langle f,\chi_\ell\rangle
=
\sum_{i=1}^N f(i)\overline{\chi_\ell(i)},
\qquad
\ell=0,\ldots,N-1.
\end{equation}
Equivalently,
\[
\widehat f=U^*f.
\]
The inverse graph Fourier transform is
\[
f
=
\sum_{\ell=0}^{N-1}\widehat f(\ell)\chi_\ell
=
U\widehat f.
\]
Moreover, Parseval's identity gives
\[
\langle f,g\rangle_{\mathbb C^N}
=
\langle \widehat f,\widehat g\rangle_{\mathbb C^N}.
\]
\subsection{Bandlimited signals and dynamical measurements on finite graphs}

We now introduce the finite-dimensional analogue of the graph
Paley-Wiener spaces. For $\omega\in[0,2]$, we define
\[
GPW_\omega(G)
=
\operatorname{span}\{\chi_\ell:\lambda_\ell\le \omega\},
\]
and, since the graph $G$ is fixed throughout this section, we simply write $GPW_\omega$ for $GPW_\omega(G)$ from now on.
Equivalently,
\[
f\in GPW_\omega
\quad\Longleftrightarrow\quad
\widehat f(\ell)=0
\quad\text{whenever }\lambda_\ell>\omega.
\]

Let
\[
\Lambda_\omega
=
\{\ell:\lambda_\ell\le \omega\}
=
\{0,\ldots,m-1\},
\]
where $m$ is the cardinality of $\Lambda_\omega$. Thus
\[
GPW_\omega
=
\operatorname{span}\{\chi_0,\ldots,\chi_{m-1}\},
\]
and every $f\in GPW_\omega$ admits the expansion
\[
f
=
\sum_{\ell\in\Lambda_\omega}
\widehat f(\ell)\chi_\ell.
\]

Since
\[
\widehat{\chi_j}=e_j,
\qquad
j=0,\ldots,N-1,
\]
where $\{e_j\}_{j=0}^{N-1}$ denotes the canonical basis of
$\mathbb C^N$, we identify
\[
\widehat{GPW_\omega}
=
\operatorname{span}\{e_0,\ldots,e_{m-1}\} \subseteq \bC^N.
\]

Let $h:[0,2]\to\mathbb C$ be a bounded function and define the graph
dynamical operator
\[
R:\ell^2(G)\longrightarrow\ell^2(G)
\]
by
\[
R
=
h(\mathcal L)
=
Uh(\Lambda)U^*
=
UM_\gamma U^*,
\]
where
\[
\gamma
=
\bigl(h(\lambda_0),\ldots,h(\lambda_{N-1})\bigr),
\]
and $M_\gamma$ denotes the diagonal multiplication operator
\[
(M_\gamma g)(i)
=
h(\lambda_i)g(i),
\qquad
i=0,\ldots,N-1.
\]
Consequently,
\[
R^*
=
UM_{\overline\gamma}U^*.
\]

Observe that $GPW_\omega$ is invariant under both $R$ and $R^*$,
so both operators can be restricted to this subspace.

For every $f,g\in GPW_\omega$ and every $n\ge0$,
\begin{align}
\langle R^n f,g\rangle
&=
\langle f,(R^*)^n g\rangle
\notag\\
&=
\langle f,UM_{\overline\gamma}^nU^*g\rangle
\notag\\
&=
\langle f,UM_{\overline\gamma}^n\widehat g\rangle
\notag\\
&=
\langle\widehat f,
M_{\overline\gamma}^n\widehat g\rangle,
\label{igualdad}
\end{align}
where $\widehat f=U^*f$.

Let
$S\subseteq V$ be a sampling set, $g_i\in GPW_\omega\ (i\in S)$,
\text{ and } $K=\{0,\dots,k-1\}$.

A signal $f\in GPW_\omega$ can be recovered stably from the
measurements
\[
\bigl\{
\langle f,(R^*)^ng_i\rangle:
n\in K,\;
i\in S
\bigr\},
\]
provided that
\[
\bigl\{
(R^*)^ng_i:
n\in K,\;
i\in S
\bigr\}
\]
is a frame for $GPW_\omega$.

Using equation \eqref{igualdad}, we get
\[
\langle f,(R^*)^ng_i\rangle
=
\langle
\widehat f,
M_{\overline\gamma}^n\widehat g_i
\rangle,
\]
and hence the above is equivalent to requiring that
\[
Q
=
\{
M_{\overline\gamma}^n\widehat g_i:
n\in K,\;
i\in S
\}
\]
is a frame for $\widehat{GPW_\omega}$.

If $\{q_{n,i}\}$ is any dual frame of $Q$, then
\[
\widehat f
=
\sum_{n\in K}
\sum_{i\in S}
\langle
\widehat f,
M_{\overline\gamma}^n\widehat g_i
\rangle
q_{n,i},
\]
and applying $U$ yields the reconstruction formula
\[
f
=
\sum_{n\in K}
\sum_{i\in S}
\langle
f,
(R^*)^ng_i
\rangle
Uq_{n,i}
=
\sum_{n\in K}
\sum_{i\in S}
\langle
R^nf,
g_i
\rangle
Uq_{n,i}.
\]

Finally, since
\[
\widehat f(j)
=
\widehat g_i(j)
=
0,
\qquad
j=m,\ldots,N-1,
\]
the operator $M_{\overline\gamma}$ may be identified with its
restriction to the first $m$ coordinates,
\[
M_{\overline\gamma}
=
\operatorname{diag}
\bigl(
h(\lambda_0),
\ldots,
h(\lambda_{m-1})
\bigr).
\]
Likewise, all vectors in
$\widehat{GPW_\omega}$ may be viewed as elements of
$\mathbb C^m$. After recovering the first $m$ Fourier coefficients of
$\widehat f$, the remaining coefficients are known to vanish, so
$\widehat f$, and hence $f$, are completely determined.

From now on, we shall identify $M_{\overline{\gamma}}$ with its restriction
to $\mathbb{C}^m$, and every vector
$\widehat{f},\widehat{g}_i\in\widehat{GPW_\omega(G)}$ with its restriction to
the first $m$ coordinates. Thus, whenever these objects are mentioned in the
sequel, their restrictions to $\mathbb{C}^m$ are understood.
 
 The case of interest is $g_i:=P_{GPW_\omega}\delta_i$, the orthogonal projection of the Dirac delta at the vertex $i$, for which $\widehat{g_i}(\ell)=\overline{\chi_\ell(i)}$ for $\ell\in\Lambda_\omega$ and $\langle R^nf,g_i\rangle=(R^nf)(i)$ is a vertex sample of the evolved signal; this is the choice made in the next Subsection~(\ref{subsec:finite-numerics}).

Since $M_{\overline\gamma}$ is a diagonal  operator,
the characterization of when the system
\[
Q
=
\{
M_{\overline\gamma}^n\widehat g_i
\}
\]
forms a frame is available from
\cite[Theorem~2.2]{UCACHA2017}, that we state here for completeness

\begin {theorem}[cf.~\cite{UCACHA2017}, Theorem 2.2]\label{DS}

Let $S \subset \{1, \dots, m\}$ and let $\{b_i: i \in S \}$ be vectors in $\bC^m$. Let $D$ be a diagonal matrix and let $r_i$ be the degree of the $D$-annihilator of $b_i$. Set $\ell_i=r_i-1$. Then  $\{D^{j}b_i: \; i\in S,  \, j=0, \dots, \ell_i\}$ is a frame of $\bC^m$ if and only if $\{P_j(b_i):i \in S\}$ is a frame of $P_j(\bC^m)$ for every eigenvalue $\lambda_j$ of $D$.
\end{theorem}

Recall that for a square matrix $Q \in \bC^{m\times m}$ the {\em $Q$-annihilator $p_b$ of a vector $b\in \bC^m$} is the monic polynomial of smallest degree, such that $p_b(Q)b \equiv 0$. 
Here $P_j$ denotes
the orthogonal projection in $\C^m$ onto the eigenspace of $D$ associated to the eigenvalue $\lambda_j$.

In our case we need to find functions  $\{g_i\in GPW_{\omega}, \; i\in S\}$
such that $\{\widehat{g_i}: i\in S\} $ restricted to $\bC^m$ satisfy the conditions of Theorem \ref{DS}.
In particular, if the values $h(\lambda_0),\dots,h(\lambda_{m-1})$ are pairwise distinct, all eigenspaces of $M_{\overline\gamma}$ are one-dimensional and Theorem~\ref{DS} shows that a single vector $g$ suffices, provided $\widehat g(\ell)\neq0$ for every $\ell\in\Lambda_\omega$; then $m-1$ iterations are needed and sufficient. For $g=P_{GPW_\omega}\delta_i$ this means $\chi_\ell(i)\ne0$ for all $\ell\in\Lambda_\omega$.

\begin{remark}
More generally, the operator $M_{\overline{\gamma}}$ may be replaced by any bounded linear
operator whose iterates satisfy the hypotheses of
\cite[Corollary~2.10]{UCACHA2017}. Consequently, the same reconstruction
result holds in this more general setting.
\end{remark}



\subsection{Analytical frame bounds and numerical results}
\label{subsec:finite-numerics}

Our goal is to understand how the quality of a finite dynamical frame
improves as we include more and more iterates of the dynamics.

\paragraph{\bf The setting.}
We work on the path graph \(P_{40}\) with bandwidth parameter
\(\omega = 0.1\). For this choice, the space of \(\omega\)-bandlimited
signals has dimension
\[
\dim GPW_\omega(P_{40}) = 6 ,
\]
so every signal we consider is determined by six coefficients.
The dynamics is the spectral multiplier
\[
  h(\lambda) = 0.95\,e^{-20\lambda},
\]
that is, the evolution operator acts on each eigenvector \(\chi_\ell\)
by multiplication by \(h(\lambda_\ell)\).

The dynamics is the spectral multiplier
\[
h(\lambda) = 0.95\,e^{-20\lambda},
\]
that is, the evolution operator acts on each eigenvector \(\chi_\ell\)
by multiplication by \(h(\lambda_\ell)\).
\paragraph{\bf Active frequencies.}
Only finitely many frequencies are present in \(GPW_\omega(P_{40})\).
Denote this \emph{active frequency set} by
\[
\Lambda_\omega = \{\ell_1,\ldots,\ell_{m_\omega}\}, 
\]
and write the associated multiplier values as
\[
h_r := h(\lambda_{\ell_r}),
\qquad r = 1,\ldots,m_\omega .
\]
So \(h_r\) is the factor by which the eigenvector \(\chi_{\ell_r}\) is damped in one step of the dynamics.

\paragraph{\bf The analysis matrix.}
Fix a sampling set \(S \subseteq V(P_{40})\) and a number of iterates
\(K \geq 1\). The experiment we perform is: sample the signal at every
vertex \(i \in S\), at every time \(n = 0,1,\ldots,K-1\). This produces
\(|S|\,K\) measurements in total.

Each measurement is a linear functional of the six unknown coefficients,
and collecting these functionals as rows gives the
\emph{analysis matrix} \(\Theta_{S,K}\). Its rows are indexed by the
space-time pair \((i,n)\) and its columns by the frequency index \(r\):
\[
\Theta_{S,K}
=
\bigl[\, \chi_{\ell_r}(i)\, h_r^{\,n} \,\bigr]_{(i,n),\,r}
\in \mathbb{C}^{|S|K \times m_\omega},
\qquad
i \in S,\quad 0 \le n < K,\quad 1 \le r \le m_\omega .
\]
The entry \(\chi_{\ell_r}(i)\,h_r^{\,n}\) has a direct interpretation:
\(\chi_{\ell_r}(i)\) is the value of the \(r\)-th eigenvector at the
sampling vertex \(i\), and \(h_r^{\,n}\) is the damping accumulated by
that eigenvector after \(n\) steps.

\paragraph{\bf A remark on ordering.}
The rows of \(\Theta_{S,K}\) may be listed in any order, for instance
vertex-by-vertex or time-by-time. Permuting rows changes neither the
rank nor the singular values of the matrix, and hence has no effect on
the frame bounds. Any convenient ordering may therefore be used.


If
\[
\widehat f_\omega
=
\bigl(
\widehat f(\ell_1),\ldots,
\widehat f(\ell_{m_\omega})
\bigr)^T,
\]
then \(\Theta_{S,K}\widehat f_\omega\) is the vector of dynamical samples
\[
\{(R^nf)(i):i\in S,\ n=0,\ldots,K-1\}.
\]

The associated frame operator is
\[
\mathcal S_{S,K}
=
\Theta_{S,K}^*\Theta_{S,K}.
\]
Its entries can be written explicitly as
\[
(\mathcal S_{S,K})_{r,s}
=
\sum_{i\in S}
\overline{\chi_{\ell_r}(i)}
\chi_{\ell_s}(i)
\Phi_K(\overline{h_r}h_s),
\]
where
\[
\Phi_K(z)
=
\sum_{n=0}^{K-1}z^n
=
\begin{cases}
\dfrac{1-z^K}{1-z},
& z\neq1,\\[2ex]
K,
& z=1.
\end{cases}
\]
Thus the frame operator is known analytically for every sampling set
\(S\) and every number of iterates \(K\).

We define
\[
A_{S,K}
=
\lambda_{\min}(\mathcal S_{S,K}),
\qquad
B_{S,K}
=
\lambda_{\max}(\mathcal S_{S,K}).
\]
Equivalently,
\[
A_{S,K}
=
\sigma_{\min}(\Theta_{S,K})^2,
\qquad
B_{S,K}
=
\sigma_{\max}(\Theta_{S,K})^2,
\]
where zero singular values are included. The sampling family is a frame
precisely when \(A_{S,K}>0\), or equivalently when
\(\Theta_{S,K}\) has full column rank. In this case,
\(A_{S,K}\) and \(B_{S,K}\) are the optimal frame bounds, and the condition
number of the frame operator is
\[
\kappa(\mathcal S_{S,K})
=
\frac{B_{S,K}}{A_{S,K}}.
\]

Let
\[
v_{i,K}
=
\bigl(
\chi_{\ell_1}(i)h_1^K,\ldots,
\chi_{\ell_{m_\omega}}(i)h_{m_\omega}^K
\bigr)
\]
be the row contributed by vertex \(i\) at the new iterate \(K\).
Adding one further iterate appends these rows to the analysis matrix and
gives
\[
\mathcal S_{S,K+1}
=
\mathcal S_{S,K}
+
\sum_{i\in S}v_{i,K}^*v_{i,K}.
\]
Since the additional term is positive semidefinite,
\[
A_{S,K+1}\geq A_{S,K},
\qquad
B_{S,K+1}\geq B_{S,K}.
\]
Thus additional iterates cannot decrease the lower frame bound. The
condition number, however, need not improve and must be examined
separately.

We compare sampling at one endpoint,
\[
S_1=\{1\},
\]
with sampling at both endpoints,
\[
S_2=\{1,40\}.
\]
We also consider the spatially distributed set
\[
S_{\mathrm d}
=
\{1,4,7,\ldots,40\},
\]
which contains fourteen sampling vertices.

For the normalized Laplacian of \(P_{40}\), the eigenvalues are
\[
\lambda_\ell
=
1-\cos\left(\frac{\pi\ell}{39}\right),
\qquad
\ell=0,\ldots,39.
\]
Since
\[
\lambda_5<0.1<\lambda_6,
\]
the active frequency set is
\[
\Lambda_\omega=\{0,1,2,3,4,5\},\qquad m_\omega=6.
\]

For these frequencies, the normalized eigenvectors may be chosen as
\[
\chi_0(i)
=
\frac{\sqrt{d(i)}}{\sqrt{78}},
\]
and
\[
\chi_\ell(i)
=
\sqrt{\frac{d(i)}{39}}
\cos\left(
\frac{\pi\ell(i-1)}{39}
\right),
\qquad
\ell=1,\ldots,5,
\]
where \(d(1)=d(40)=1\) and \(d(i)=2\) for
\(2\leq i\leq39\) (these are the eigenpairs of the normalized Laplacian of the path graph; they can be verified directly from \eqref{normL}).

The multiplier $h$ is injective, so the values $h_0,\dots,h_5$ are distinct, and $\chi_\ell(1)\neq0$ for all $\ell$ (in fact $\chi_0(1)=\sqrt{1/78}$, while $\chi_\ell(1)=\sqrt{1/39}$ for $\ell=1,\dots,5$; similarly $\chi_\ell(40)=\pm\sqrt{1/78}$ or $\pm\sqrt{1/39}$); by Theorem~\ref{DS} and the discussion following it, sampling at one endpoint yields a frame exactly when $K\ge m_\omega=6$, and at both endpoints when $K\geq3$. The numerical results below confirm this and quantify the conditioning.


On the vertices of \(S_{\mathrm d}\), which have the form
\[
i=1+3j,
\qquad
j=0,\ldots,13,
\]
the discrete cosine orthogonality relations give
\[
\sum_{i\in S_{\mathrm d}}
\overline{\chi_r(i)}\chi_s(i)
=
\frac13\delta_{rs},
\qquad
r,s=0,\ldots,5.
\]
Consequently, the frame operator associated with the distributed
configuration is diagonal:
\[
\mathcal S_{S_{\mathrm d},K}
=
\frac13
\operatorname{diag}
\left(
\Phi_K(|h_0|^2),\ldots,\Phi_K(|h_5|^2)
\right).
\]
The corresponding optimal frame bounds are therefore
\[
A_{S_{\mathrm d},K}
=
\frac13
\min_{0\leq r\leq5}
\Phi_K(|h_r|^2),
\]
and
\[
B_{S_{\mathrm d},K}
=
\frac13
\max_{0\leq r\leq5}
\Phi_K(|h_r|^2).
\]

For the multiplier considered here, \(h(\lambda)\) is positive and strictly
decreasing. Hence
\[
A_{S_{\mathrm d},K}
=
\frac13
\frac{
1-h(\lambda_5)^{2K}
}{
1-h(\lambda_5)^2
},
\]
whereas
\[
B_{S_{\mathrm d},K}
=
\frac13
\frac{
1-h(\lambda_0)^{2K}
}{
1-h(\lambda_0)^2
}.
\]

In particular, for \(K=1\),
\[
\mathcal S_{S_{\mathrm d},1}
=
\frac13I_6,
\]
and therefore
\[
A_{S_{\mathrm d},1}
=
B_{S_{\mathrm d},1}
=
\frac13,
\qquad
\frac{B_{S_{\mathrm d},1}}
     {A_{S_{\mathrm d},1}}
=
1.
\]
Thus the distributed sampling family is exactly tight. After multiplying
its vectors by \(\sqrt{3}\), it becomes a Parseval frame. The numerical
relative residual
\[
\frac{
\left\|
\mathcal S_{S_{\mathrm d},1}-\frac13I_6
\right\|_{\mathrm F}
}{
\left\|\mathcal S_{S_{\mathrm d},1}\right\|_{\mathrm F}
}
\approx
7.1\times10^{-16}
\]
is consistent with this exact identity.

At \(K=20\), the analytical formulas give
\[
A_{S_{\mathrm d},20}
\approx
0.346,
\qquad
B_{S_{\mathrm d},20}
\approx
2.979,
\]
and
\[
\frac{B_{S_{\mathrm d},20}}
     {A_{S_{\mathrm d},20}}
\approx
8.61.
\]
Hence the lower frame bound increases, but the condition number moves away
from its optimal value \(1\). This shows analytically that additional
iterates do not preserve tightness for this configuration.

For the one- and two-endpoint configurations, the entries of the frame
operator are still given explicitly by the formula above, but its extremal
eigenvalues do not reduce to equally simple closed expressions. We
therefore compute them numerically by diagonalizing the corresponding
\(6\times6\) Hermitian matrices.

Figure~\ref{fig:frame-bounds-iterations} shows the lower frame bound and
the condition number of the frame operator as functions of the number of
iterates. Both quantities are displayed on logarithmic vertical scales in
order to make the behavior of the three sampling configurations visible
over their different orders of magnitude.

\begin{figure}[htbp]
    \centering
    \includegraphics[width=0.85\textwidth]
    {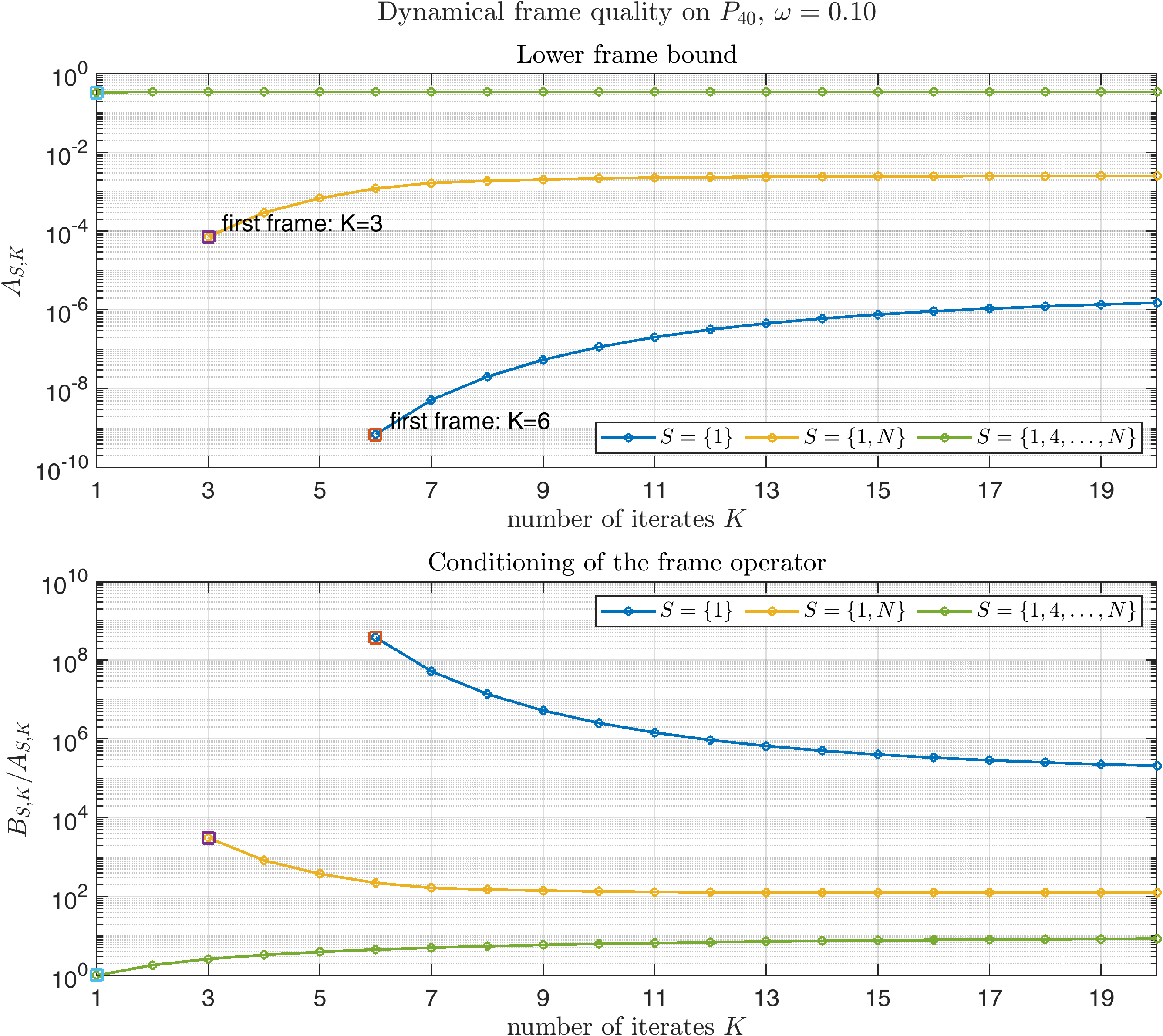}
    \caption{Quality of the finite dynamical frames on \(P_{40}\), with
    \(\omega=0.1\) and \(h(\lambda)=0.95e^{-20\lambda}\). The upper panel
    shows the optimal lower frame bound \(A_{S,K}\), while the lower panel
    shows the condition number \(B_{S,K}/A_{S,K}\) of the frame operator.
    Both vertical axes use logarithmic scales. The distributed sampling set
    \(S_{\mathrm d}=\{1,4,7,\ldots,40\}\) gives an exact tight frame at
    \(K=1\).}
    \label{fig:frame-bounds-iterations}
\end{figure}

For the single-endpoint configuration \(S_1=\{1\}\), the analysis matrix
first attains full column rank at \(K=6\). At this value,
\[
A_{S_1,6}
\approx
7.02\times10^{-10},
\qquad
B_{S_1,6}
\approx
2.64\times10^{-1},
\]
and hence
\[
\frac{B_{S_1,6}}{A_{S_1,6}}
\approx
3.77\times10^8.
\]
The sampling family therefore spans the active spectral space, but the
resulting frame is extremely ill-conditioned. Increasing the number of
iterates to \(K=20\) gives
\[
A_{S_1,20}
\approx
1.53\times10^{-6},
\qquad
B_{S_1,20}
\approx
3.19\times10^{-1},
\]
and
\[
\frac{B_{S_1,20}}{A_{S_1,20}}
\approx
2.09\times10^5.
\]
Thus the additional iterates substantially improve both the lower frame
bound and the conditioning, although the frame remains poorly conditioned.
This is a Vandermonde-type phenomenon: the rows contributed by a single vertex
are $(\chi_r(1)h_r^n)_r$, and the conditioning is governed by the spread of
$h_0,\dots,h_5$; with $h(\lambda)=0.95e^{-20\lambda}$ one has $h_5\approx0.19$, so
the last columns are tiny. The decay rate $20$ in $h$ therefore strongly
influences the numbers reported here; a slower decay improves the conditioning
of the single-vertex scheme at the price of a slower gain from additional
iterates.

For the two-endpoint configuration \(S_2=\{1,40\}\), full column rank is
already obtained at \(K=3\). In this case,
\[
A_{S_2,3}
\approx
7.30\times10^{-5},
\qquad
B_{S_2,3}
\approx
2.28\times10^{-1},
\]
and
\[
\frac{B_{S_2,3}}{A_{S_2,3}}
\approx
3.12\times10^3.
\]
At their respective first full-rank values, the one- and two-endpoint
systems contain the same total number of sampling vectors:
\[
|S_1|\cdot6
=
|S_2|\cdot3
=
6.
\]
Nevertheless,
\[
\frac{A_{S_2,3}}{A_{S_1,6}}
\approx
1.04\times10^5.
\]
Thus, with the same total number of measurements, distributing the samples
over two vertices gives a lower frame bound approximately \(10^5\) times
larger than collecting all temporal samples at a single vertex.

At \(K=20\), the two-endpoint configuration gives
\[
A_{S_2,20}
\approx
2.55\times10^{-3},
\qquad
B_{S_2,20}
\approx
3.28\times10^{-1},
\]
and
\[
\frac{B_{S_2,20}}{A_{S_2,20}}
\approx
1.29\times10^2.
\]
The two-endpoint configuration is therefore substantially better
conditioned than the one-endpoint configuration throughout the considered
range of iterates.

We include the distributed configuration as a reference benchmark, but we do not interpret it as a direct comparison with the single-vertex experiments. Indeed, it uses fourteen spatial
samples at \(K=1\), whereas the first full-rank one- and two-endpoint
systems each contain six measurements. It nevertheless demonstrates that
a suitable spatial distribution of the sampling vertices can produce an
optimally conditioned frame.

The principal values are summarized in
Table~\ref{tab:frame-bounds-iterations}.

\begin{table}[htbp]
\centering
\begin{tabular}{c|c|c|c|c}
\hline
Sampling set \(S\)
& \(K\)
& \(A_{S,K}\)
& \(B_{S,K}\)
& \(B_{S,K}/A_{S,K}\)
\\
\hline
\(\{1\}\)
& \(6\)
& \(7.02\times10^{-10}\)
& \(2.64\times10^{-1}\)
& \(3.77\times10^{8}\)
\\
\(\{1\}\)
& \(20\)
& \(1.53\times10^{-6}\)
& \(3.19\times10^{-1}\)
& \(2.09\times10^{5}\)
\\
\(\{1,40\}\)
& \(3\)
& \(7.30\times10^{-5}\)
& \(2.28\times10^{-1}\)
& \(3.12\times10^{3}\)
\\
\(\{1,40\}\)
& \(20\)
& \(2.55\times10^{-3}\)
& \(3.28\times10^{-1}\)
& \(1.29\times10^{2}\)
\\
\(\{1,4,7,\ldots,40\}\)
& \(1\)
& \(1/3\)
& \(1/3\)
& \(1\)
\\
\(\{1,4,7,\ldots,40\}\)
& \(20\)
& \(3.46\times10^{-1}\)
& \(2.98\)
& \(8.61\)
\\
\hline
\end{tabular}
\vspace{3mm}
\caption{Optimal frame bounds and frame-operator condition numbers for the
three dynamical sampling configurations on \(P_{40}\). The bounds for the
distributed configuration follow from the analytical formulas, while the
endpoint bounds are obtained by diagonalizing the corresponding explicit
frame matrices.}
\label{tab:frame-bounds-iterations}
\end{table}

In all three configurations, additional iterates do not decrease the lower
frame bound, in agreement with the monotonicity of the frame operators.
For the one- and two-endpoint configurations, they also substantially
improve the conditioning. In contrast, the distributed configuration is
already exactly tight at \(K=1\), and additional iterates move its
condition number away from the optimal value \(1\). Thus increasing
temporal redundancy does not necessarily improve tightness or conditioning.

\newpage
\section{Appendix}\label{apendix}

The following Lemma provides an orthonormal basis of $GPW_\omega(\Z).$
\begin{lemma*}\label{repro}
For a fixed $\ell\in\bN$ such that $\omega=\varphi(\frac{1}{2\ell})=1-\cos(\frac{\pi}{\ell}),$ and any  $j\in\bZ,$ define
$$g_{j}(m)=\frac{1}{\sqrt{\ell}}\,\sinc\left(\frac{m}{\ell}-j\right),\quad m\in\bZ$$
then $\{g_j\}_{j\in\mathbb{Z}}$ is an orthonormal basis of $GPW_{\omega}(\Z),$ and, for any $f\in GPW_\omega(\bZ),$
\begin{equation}\label{repker}
\langle f,\,g_{j}\rangle_{\ell^2(\bZ)}=\sqrt{\ell}\,f(\ell j).
\end{equation}
\end{lemma*}
\begin{proof}
If  $f\in GPW_\omega(\bZ),$ then $\mcU f$ has support in $[0,\omega]$ which means $\cF{f}$ has support in $\Omega=[-\frac{1}{2\ell},\frac{1}{2\ell}],$ where $$\omega=\varphi(\frac{1}{2\ell})=1-\cos(\frac{\pi}{\ell}).$$
By \eqref{chiinv2} we have
\begin{eqnarray*}
\sinc(\frac{m}{\ell}-j)&=&\ell\int_{-1/2}^{1/2} \chi_{\ds \Omega}(\xi) e^{-2\pi i \xi(m-\ell j)} \, d\xi=\ell\int_{-1/2}^{1/2} \chi_{\ds \Omega}(\xi)\, e^{2\pi i \ell j\xi}\,e^{-2\pi i m\xi}  \, d\xi\\ \\
& =&
\ell \cF^{-1}\left(\chi_{\ds \Omega}\, e^{2\pi i \ell j\cdot}\right) (m),
\end{eqnarray*}
hence
$$g_j=\sqrt{\ell}\,\cF^{-1}\left(\chi_{\ds \Omega}\, e^{2\pi i \ell j\cdot}\right)\Leftrightarrow
\cF(g_j)(\xi)=\sqrt{\ell}\,\chi_{\ds \Omega}(\xi)\, e^{2\pi i \ell j\xi}.$$
Therefore $g_j\in GPW_{\omega}(\Z).$
We compute
\begin{eqnarray*}
 \langle g_p,\,g_{j}\rangle_{\ell^2(\bZ)}&=&\langle \cF(g_p),\,\cF(g_{j})\rangle_{L^2(\mathbb{T})}
=\ell\,\int_{\mathbb{T}}\chi_{\ds \Omega}(\xi)\, e^{2\pi i \ell (p-j)\xi}\, d\xi=
\ell\,\int_{-1/2\ell}^{1/2\ell}e^{2\pi i \ell (p-j)\xi}\,d\xi\\ \\
&=&\int_{-1/2}^{1/2}e^{2\pi i  (p-j)\xi}\, d\xi=\delta_{p,j},
\end{eqnarray*}
therefore $\{g_j\}_{j\in\mathbb{Z}}$ is orthonormal.
To see that it is actually a basis it is sufficient to show \eqref{repker}, indeed
 by Theorem \ref{ShannotheoremZ}, and \eqref{ShannonZeq}, if $f\in GPW_\omega(\bZ),$
we have
$$
    f(m) = \sum_{j \in \mathbb{Z}} f(\ell j)\sinc(m/\ell - j),\quad m\in\bZ,
$$
so that $0=\langle f,\,g_{j}\rangle_{\ell^2(\bZ)}$ for all $j$ implies $f(\ell j)=0$ and
$f\equiv 0.$

Now
\begin{align*}
& \langle f,\,g_{j}\rangle_{\ell^2(\bZ)}=\frac{1}{\sqrt{\ell}}\,\sum_{m\in\mathbb{Z}}f(m)\,\sinc(m/\ell - j)=
\frac{1}{\sqrt{\ell}}\,\sum_{m\in\mathbb{Z}}\sum_{p \in \mathbb{Z}} f(\ell p)\sinc(m/\ell - p)\,\sinc(m/\ell - j)\\ \\
&=
\frac{1}{\sqrt{\ell}}\,\sum_{p \in \mathbb{Z}} f(\ell p)\sum_{m\in\mathbb{Z}}\sinc(m/\ell - p)\,\sinc(m/\ell - j)=
\sqrt{\ell}\,\sum_{p \in \mathbb{Z}} f(\ell p)\langle g_p,g_j\rangle_{\ds \ell^2(\bZ)}\\ \\
&=
\sqrt{\ell}\,\sum_{p \in \mathbb{Z}} f(\ell p)\,\delta_{p,j}
=
\sqrt{\ell}\,f(\ell j).
\end{align*}
\end{proof}

\section*{Acknowledgments}

The authors acknowledge the support of the ICTP--INdAM Collaborative Grants
and Research in Pairs Programme 2024, under which part of this project was
developed. The authors also thank the Abdus Salam International Centre for
Theoretical Physics (ICTP) for its support and hospitality, which made possible a fruitful collaboration between the Argentine
group at UBA--CONICET and the Italian INdAM research group.

IMB and SS were supported by the Italian MUR project PRIN 2022,  20227TRY8H ``TIme-varying signals on Graphs: REal and COmplex methods'' (TIGRECO). 
IMB was supported by the INdAM--GNCS Project, CUP E53C25002010001. This
research was carried out within RITA (Research ITalian network on
Approximation) and the UMI Group TAA (Approximation Theory and Applications). CC and UM acknowledge grants  PIP 202287/22 (CONICET), and UBACyT 2022-154 (UBA).

\section{Conflict of Interest}

The authors declare no conflict of interest.

\end{document}